\documentclass[13pt, reqno]{amsart}
\usepackage{dsfont}
\usepackage{amsmath}
\usepackage{amscd}
\usepackage{amsthm}
\usepackage{amssymb} \usepackage{latexsym}
\usepackage{eufrak}
\usepackage{euscript}
\usepackage{epsfig}
\usepackage{graphics}
\usepackage{array}
\usepackage{enumerate}
\usepackage{color}
\usepackage{wasysym}
\usepackage{pdfsync}
\usepackage{stmaryrd}
\usepackage{url}
\usepackage[font=small, labelformat=simple]{caption}

\newcommand{\bel}[1]{\begin{equation*}\label{#1}}
	
	\newcommand{\be}{\begin{equation}}

		\newcommand{\ba}{\begin{eqnarray}}
			\newcommand{\ea}{\end{eqnarray}}

		\newcommand{\qe}{\end{equation}}
	\newcommand{\R}{{\mathbb R}}
	
	\newcommand{\Z}{{\mathbb Z}}

	\newcommand{\supp}{{\mathrm{supp}}}

	\newcommand{\eg}{\begin{example}}
		\newcommand{\egd}{\end{example}}
	\newcommand{\tm}{\begin{thm}}
		\newcommand{\tmd}{\end{thm}}
	\newcommand{\co}{\begin{coro}}
		\newcommand{\cod}{\end{coro}}
	\newcommand{\enu}{\begin{enumerate}}
		\newcommand{\enud}{\end{enumerate}}
	\newcommand{\rmk}{\begin{rem}}
		\newcommand{\rmkd}{\end{rem}}
	
	\theoremstyle{theorem}
	\newtheorem{thm}{Theorem}[section]

	\theoremstyle{example}
	\newtheorem{example}[thm]{Example}
	\newtheorem{coro}[thm]{Corollary}
	\theoremstyle{lemma}
	\newtheorem{lemma}[thm]{Lemma}
	\theoremstyle{definition}
	\newtheorem{defi}[thm]{Definition}
	\theoremstyle{proof}
	
	\theoremstyle{remark}
	\newtheorem{rem}[thm]{Remark}
	\theoremstyle{remark}

	\UseRawInputEncoding

\begin{document}

		\title[Continuum limit of 3D fractional nonlinear Schr\"{o}dinger equation]{Continuum limit of 3D fractional nonlinear Schr\"{o}dinger equation }

		\author{Jiajun Wang}
		\address{Jiajun Wang: School of Mathematical Sciences,
			Fudan University, Shanghai 200433, China.}
		\email{21300180146@m.fudan.edu.cn}
		
		\begin{abstract}
			In this paper, we investigate the continuum limit theory of the fractional nonlinear Schr\"{o}dinger equation in dimension 3. We show that the solution of discrete fractional nonlinear Schr\"{o}dinger equation on $h\mathbb{Z}^{3}$ will converge strongly in $L^{2}$ to the solution of fractional nonlinear Schr\"{o}dinger equation on $\mathbb{R}^{3}$, when $h\to 0$. The key is proving the uniform-in-$h$ Strichartz estimate for discrete fractional nonlinear Schr\"{o}dinger equation, by using the uniform estimate of oscillatory integral and Newton polyhedron techniques. 
		\end{abstract}

		\maketitle
		\numberwithin{equation}{section}
		\section{Introduction}
		    The discrete nonlinear Schr\"{o}dinger equation has gotten more and more attention from both physics and mathematics.
		    
		     In physics,  discrete nonlinear Schr\"{o}dinger equation is served as a model in amorphous materials \cite{20,21} and denaturation of DNA double strand \cite{22,23}. For more information about discrete nonlinear Schr\"{o}dinger equation's applications, we refer to \cite{24,25}. 
		     
		     In mathematics, one often discretizes the PDEs for simulation purpose. Therefore, both for theoretical and practical purpose, proving that the solutions of discrete PDEs will converge to the solutions of their continuous counterparts, as the mesh tends to  $0$, is necessary and important. 
		     
		    Besides,  the topic `` continuum limit of fractional nonlinear Schr\"{o}dinger equation " has garnered great interest from many mathematicians. In \cite{3}, K. Kirkpatrick, E. Lenzmann and G. Staffilani proved the 1D continuum limit with long-range lattice interaction, in a weak convergence sense. In \cite{1}, Y. Hong and C. Yang strengthened the result into strong convergence, by using the Littlewood-Paley theory on lattice. Recently, B. Choi and A. Aceves extended the continuum limit theory to dimension 2 \cite{8}. In this paper, we will further extend the theory to dimension 3.
		     
			We first introduce fractional nonlinear Schr\"{o}dinger equation (FNLS)
		\begin{equation}\label{FNLS}
			\left\{
			\begin{aligned}
				& i\partial_{t} u_{h}(x,t)=(-\Delta_{h})^{\frac{\alpha}{2}} u_{h}(x,t) +\mu |u_{h}|^{p-1}u_{h},  \\
				& u_{h}(x,0) = u_{0,h}(x), \quad (x,t)\in h\Z^d\times \R,
			\end{aligned}
			\right.
		\end{equation}
		and its continuum limit 
		\begin{equation}\label{CFNLS}
			\left\{
			\begin{aligned}
				& i\partial_{t} u(x,t)=(-\Delta)^{\frac{\alpha}{2}} u(x,t) +\mu |u|^{p-1}u,  \\
				& u(x,0) = u_{0}(x), \quad (x,t)\in \mathbb{R}^d\times \R,
			\end{aligned}
			\right.
		\end{equation}
		where $0<h\le 1$, $u_{h}: h\Z^{d}\to \mathbb{C}$, $u: \mathbb{R}^{d}\to \mathbb{C}$, $\alpha\in (0,2]-\lbrace 1\rbrace$, $p>1,\mu=\pm 1$.
		
		Before we state our main result, we shall briefly recall some basic concepts and notations.
		The discrete Laplacian on $h\Z^{d}$ is defined as
		\begin{equation*}
			\Delta_{h}u:=\sum_{i=1}^{d}\dfrac{u(x+he_{i})+u(x-he_{i})-2u(x)}{h^{2}}, \quad x\in h\Z^{d},
		\end{equation*}
		where $u:h\Z^{d}\to \mathbb{C}$ and $\lbrace e_{i}\rbrace_{i=1}^{d}$ is the standard basis in $\Z^{d}$. 
		
		To properly define the fractional discrete Laplacian, we need to use the discrete Fourier transform and its inverse, which are defined as follows
		\begin{equation*}
			(\mathcal{F}_{h}u)(\xi):= h^{d}\sum_{x\in h\Z^{d}}u(x)e^{-ix\cdot\xi}, \quad \xi\in \frac{\mathbb{T}^{d}}{h},
		\end{equation*}
		\begin{equation*}
			(\mathcal{F}_{h}f)^{-1}(x):=\dfrac{1}{(2\pi)^{d}}\int_{\frac{\mathbb{T}^{d}}{h}}f(\xi)e^{ix\cdot\xi}d\xi, \quad x\in h\Z^{d},
		\end{equation*}
		where $\mathbb{T}^{d}=[-\pi,\pi]^{d}$. Then the fractional discrete Laplacian is defined through Fourier multiplier as follows
		\begin{equation*}
			(-\Delta_{h})^{\frac{\alpha}{2}}u:=\mathcal{F}_{h}^{-1}\left\lbrace\left[\sum_{i=1}^{d}\frac{4}{h^{2}}\sin^{2}(\frac{h\xi_{i}}{2})\right]^{\frac{\alpha}{2}}\mathcal{F}_{h}u\right\rbrace.
		\end{equation*}
		We also denote the propagator $U_{h}(t):=e^{-it(-\Delta_{h})^{\frac{\alpha}{2}}}$ and, similarly, $U(t):=e^{-it(-\Delta)^{\frac{\alpha}{2}}}$.
		 
		To set our discussion in a suitable regime, we introduce some important function spaces.
		For $1\le p\le \infty$, the $L_{h}^{p}$-space is defined as
		    \begin{equation*}
		    	L_{h}^{p}:=\left\lbrace u:h\Z^{d}\to \mathbb{C}\bigg| h^{\frac{d}{p}}(\sum_{x\in h\Z^{d}}|u(x)|^{p})^{\frac{1}{p}}<\infty\right\rbrace,
		    \end{equation*}
		    with the $L_{h}^{p}$-norm 
		    \begin{equation*}
		    	\|u\|_{L_{h}^{p}}:= h^{\frac{d}{p}}(\sum_{x\in h\Z^{d}}|u(x)|^{p})^{\frac{1}{p}}.
		    \end{equation*}
	  Following the definition from \cite{1} or \cite{2}, the Sobolev norm on $h\Z^{d}$ is defined as
	  \begin{equation*}
	  	\|u\|_{W_{h}^{s,p}}:=\|\langle \nabla_{h}\rangle^{s}u\|_{L_{h}^{p}}, \quad \langle \nabla_{h}\rangle^{s}u:=\mathcal{F}_{h}^{-1}\left\lbrace\langle \xi\rangle^{s}\mathcal{F}_{h}u\right\rbrace,
	  \end{equation*}
	  where $s\in \mathbb{R}$, $p\in (1,\infty)$ and $\langle \xi\rangle:=\left(1+|\xi|^{2}\right)^{\frac{1}{2}}$. Besides, we conventionally denote $H_{h}^{s}:=W_{h}^{s,2}$.
	  
	  Next, we introduce some necessary operators that connect the discrete PDEs and their continuum limit.
	  
	  The discretization operator $d_{h}:L^{2}(\mathbb{R}^{d})\to L_{h}^{2}$ is defined as 
	  \begin{equation*}
	  	d_{h}f(x)=\frac{1}{h^{d}}\int_{x+[0,h)^{d}}f(y)dy, \quad x\in h\Z^{d}.
	  \end{equation*}
	  The linear interpolation operator $p_{h}:L_{h}^{2}\to L^{2}(\mathbb{R}^{d})$ is defined as
	  \begin{equation*}
	  	p_{h}u(y)=u(x)+\sum_{i=1}^{d}\dfrac{u(x+he_{i})-u(x)}{h}(y_{i}-x_{i}), \quad y\in x+[0,h)^{d}, \; x\in h\Z^{d},
	  \end{equation*}
	  where $y=(y_{1},\cdots, y_{d})$, $x=(x_{1},\cdots, x_{d})$.
	  
	  Now we state our main theorem as follows,
	  \begin{thm}\label{main}
	  	For $3\le p<5$, $\frac{3(p-1)}{p+1}<\alpha<2$, $u_{0}\in H^{\frac{\alpha}{2}}(\mathbb{R}^{3})$ and $u_{0,h}=d_{h}u_{0}$, we have $u\in C([0,T];H^{\frac{\alpha}{2}}(\mathbb{R}^{3}))$ and $u_{h}\in C([0,T];L_{h}^{2})$ are the solutions of equations (\ref{CFNLS}) and (\ref{FNLS}), for some $T=T(\|u_{0}\|_{H^{\frac{\alpha}{2}}(\mathbb{R}^{3})})>0$. Then we have
        the	following strong convergence
        \begin{equation*}
	  		\|p_{h}u_{h}(t)-u(t)\|_{L^{2}(\mathbb{R}^{d})}\le C_{1}h^{\frac{\alpha}{2+\alpha}}e^{C_{2}t}, \quad t\in [0,T],
	  	\end{equation*}
	  	where $C_{1}, C_{2}$ is positive constants that only depend on $\|u_{0}\|_{H^{\frac{\alpha}{2}}(\mathbb{R}^{3})}$.
	  \end{thm}
	  \begin{rem}
	  	The range of $\alpha$ is energy-subcritical. To be more specific, \cite{3} shows that the above well-posed solutions conserve the energy below
	  	\begin{equation*}
	  		E[u(t)]:=\frac{1}{2}\int_{\mathbb{R}^{d}}\Big||\nabla|^{\frac{\alpha}{2}}u\Big|^{2}dx+\frac{\mu}{p+1}\int_{\mathbb{R}^{2}}|u|^{p+1}dx,
	  	\end{equation*} 
	  	\begin{equation*}
	  		E_{h}[u_{h}(t)]=\frac{1}{2}\|(-\Delta_{h})^{\frac{\alpha}{2}}u_{h}(t)\|_{L_{h}^{2}}^{2}+\frac{\mu}{p+1}\|u_{h}(t)\|_{L_{h}^{p+1}}^{p+1}.
	  	\end{equation*}
	  	On the other hand, direct scaling argument shows that the Sobolev-critical regularity is 
	  	\begin{equation*}
	  		s_{c}:=\frac{d}{2}-\frac{\alpha}{p-1},
	  	\end{equation*}
	  	and the above range ``$\frac{3(p-1)}{p+1}<\alpha$" comes from ``$\frac{\alpha}{2}>s_{c}$", which is called energy-subcritical.
	  \end{rem}
	  \begin{rem}
	  	From the proof of Theorem \ref{main}, we will see explicit expressions of $T,C_{1}, C_{2}$. For example, we have, for some suitable $q_{0}>p-1$,
	  	\begin{equation*}
	  		T\simeq \|u_{0}\|_{H^{\frac{\alpha}{2}}}^{-\frac{1}{\frac{1}{p-1}-\frac{1}{q_{0}}}}.
	  	\end{equation*}
	  \end{rem}
	  
	  We organize this paper as follows. In section 2, we introduce some important lemmas, especially Lemma \ref{3.4}, and prove the desired Theorem \ref{main}. In section 3, we complete the proof of Lemma \ref{3.4}. In the Appendix, we will briefly recall the uniform estimate of oscillatory integral and related Newton polyhedron method.
	  
	  	\noindent
	  \textbf{Notation.}
	  \begin{itemize}
	  	\item By $u\in C^{k}([0,T]; B)( or \; L^{p}([0,T];B))$ for a Banach space $B,$ we mean $u$ is a $C^{k}(or \; L^{p})$ map from $[0,T]$ to $B;$ see Chapter 5 in \cite{10}.
	  	\item By $A\lesssim B$ (resp. $A\simeq B$), we mean there is a positive constant $C$, such that $A\le CB$ (resp. $C^{-1}B\le A \le C B$). If the constant $C$ depends on $p,$ then we write $A\lesssim_{p}B$ (resp. $A\simeq_{p} B$). 
	  	\item By $A\approx B$, we mean $|A-B|\ll1$.
	  	\item By $\#(S)$ for a set $S$, we mean the cardinality of $S$.
	  	\item By $B_{\mathbb{R}^{d}}(\xi,r)$(or $B_{\mathbb{C}^{d}}(\xi,r)$), we mean the open ball in $\mathbb{R}^{d}$ (or $\mathbb{C}^{d}$) with the center $\xi$ and radius $r$. $\overline{B}_{\mathbb{R}^{d}}(\xi,r)$ (or $\overline{B}_{\mathbb{C}^{d}}(\xi,r)$) is its closure.
	  	\item By $u\in\mathcal{O}(\Omega)$, we mean $u$ is holomorphic on $\Omega$. $u\in C^{k}(\Omega)$ means $u$'s $k$-th order derivatives are continuous on $\Omega$.
	  \end{itemize}
		\section{Proof of Theorem \ref{main}}
	  Before establishing the continuum limit theory of (FNLS), we shall first prove the well-posedness of (FNLS), which is built by the following lemma from \cite{4}.
	  \begin{lemma}\label{wellposedness}
	  	For $s>s_{c}$, $p\ge3$, $d\ge2$, the FNLS (\ref{CFNLS}) is locally well-posed in $H^{s}(\mathbb{R}^{d})$. For $p>1$, $d\ge 1$, the FNLS (\ref{FNLS}) is globally well-posed in $L_{h}^{2}$. 
	  \end{lemma}
	  \begin{proof}
	  	The proof can be found in \cite{4}.
	  \end{proof}
	   Another lemma is about some basic properties of discretization operator $d_{h}$ and linear interpolation operator $p_{h}$.
	   \begin{lemma}\label{discretization}
	   	For $1\le s\le 1$, $p>1$, $d\ge1$, we have the following estimates
	   	\begin{equation*}
	   		(I): \; \|d_{h}f\|_{H_{h}^{s}}\lesssim \|f\|_{H^{s}(\mathbb{R}^{d})};
	   	\end{equation*}
	   	\begin{equation*}
	   		(II):\; \|p_{h}u\|_{H^{s}(\mathbb{R}^{d})}\lesssim\|u\|_{H_{h}^{s}};
	   	\end{equation*}
	   	\begin{equation*}
	   		(III): \; \|p_{h}d_{h}f-f\|_{L^{2}(\mathbb{R}^{d})}\lesssim h^{s}\|f\|_{H^{s}(\mathbb{R}^{d})};
	   	\end{equation*}
	   	\begin{equation*}
	   			(IV): \; \|p_{h}U_{h}(t)u_{0,h}-U(t)u_{0}\|_{L^{2}(\mathbb{R}^{d})}\lesssim \langle t \rangle h^{\frac{s}{1+s}}\|u_{0}\|_{H^{s}(\mathbb{R}^{d})}, \quad u_{0,h}=d_{h}u_{0};  	
	    \end{equation*}
	    \begin{equation*}
	    	(V): \; \|p_{h}\left(|u_{h}|^{p-1}u_{h}\right)-|p_{h}u_{h}|^{p-1}p_{h}u_{h}\|_{L^{2}(\mathbb{R}^{d})}\lesssim h^{s}\|u_{h}\|_{L_{h}^{\infty}}^{p-1}\|u_{h}\|_{H_{h}}^{s}.
	    \end{equation*}
	   \end{lemma} 
	  \begin{proof}
	  	See e.g. \cite{1} and \cite{3}.
	  \end{proof}
	  The most key idea in proving Theorem \ref{main} is establishing uniform Strichartz estimate for the propagator $U_{h}(t)$. Before we introduce it, we need to use Littlewood-Paley theory on the lattice $h\Z^{d}$, which is initially established in \cite{5}.
	  
	  Take radial cutoff $\phi\in C_{c}^{\infty}(\R^{d})$, with $0\le\phi\le1$, $\supp(\phi)\subseteq \lbrace |\xi|\le 2\pi\rbrace$ and $\phi \equiv1 $ on $\lbrace |\xi|\le \pi\rbrace$. Then let $\eta(\xi):=\phi(\xi)-\phi(2\xi)$ and dyadic number $N\le 1$, the Littlewood-Paley projections is defined as 
	  \begin{equation*}
	  	P_{N}f=P_{N,h}f:=\mathcal{F}^{-1}\left \lbrace \eta(\frac{h\xi}{N})\mathcal{F}f\right\rbrace,
	  \end{equation*}
	  where $\mathcal{F}$ is Fourier transform on $\mathbb{R}^{d}$. As a result, for $u:h\Z^{d}\to \mathbb{C}$, we have the identity
	  \begin{equation*}
	  	\sum_{N\le1}P_{N}u=u,
	  \end{equation*}
	 as we can regard the frequency space $\frac{\mathbb{T}^{d}}{h}$ belongs to $\mathbb{R}^{d}$.
	 
	 Now we can state the most essential but difficult lemma as follows.
	 \begin{lemma}\label{3.4}
	 	For $d=3$, $\alpha\in (1,2)$, $h\in (0,1]$, $N\le 1$, there exists constant $C=C(\alpha)>0$ s.t.
	 	\begin{equation*}
	 		\|U_{h}(t)P_{N}u\|_{L_{h}^{\infty}}\le C(\alpha)\left(\dfrac{N}{h}\right)^{3-\alpha}|t|^{-1}\|u\|_{L_{h}^{1}}.
	 	\end{equation*}
	 \end{lemma} 
	  Assuming the Lemma \ref{3.4}, we can derive uniform Strichartz estimate on the lattice.
	  \begin{thm}\label{uniform}
	  
	  Under the hypothesis of Lemma \ref{3.4}, we have the uniform Strichartz estimate below
	  	\begin{equation*}
	  		\|U_{h}(t)u\|_{L_{t}^{q}L_{h}^{r}}\lesssim_{q,r}\left\| |\nabla_{h}|^{(3-\alpha)(\frac{1}{2}-\frac{1}{r})}u\right\|_{L_{h}^{2}},
	  	\end{equation*}
	  	where $(q,r)$ is Strichartz admissible pair, satisfying
	  	\begin{equation*}
	  		\frac{1}{q}+\frac{1}{r}=\frac{1}{2}, \quad q\in [2,\infty],\; r\in[2,\infty). 
	  	\end{equation*}
	  \end{thm}
	  \begin{proof}
	  	Consider the operator $\widetilde{U}(t):=P_{N}U_{h}\left((\frac{N}{h})^{(3-\alpha)(\frac{1}{2}-\frac{1}{r})}t\right)\widetilde{P_{N}}$, where $\widetilde{P_{N}}:=P_{\frac{N}{2}}+P_{N}+P_{2N}$. Directly, $\left \lbrace \widetilde{U}(t)\right\rbrace_{t\in\mathbb{R}}$ meets the requirement of Theorem 1.2 in \cite{6}, which yields
	  	\begin{equation*}
	  		\|U_{h}(t)P_{N}u\|_{L_{t}^{q}L_{h}^{r}}\lesssim_{q,r}\left(\frac{N}{h}\right)^{(3-\alpha)(\frac{1}{2}-\frac{1}{r})}\|P_{N}u\|_{L_{h}^{2}}\simeq\left\|P_{N}|\nabla_{h}|^{(3-\alpha)(\frac{1}{2}-\frac{1}{r})}u\right\|_{L_{h}^{2}}.
	  	\end{equation*}
	  	Then invoking Littlewood-Paley inequality on the lattice (see e.g. \cite{5}), we get
	  	\begin{equation*}
	  		\|U_{h}(t)u\|_{L_{t}^{q}L_{h}^{r}}\simeq\left(\sum_{N\le1}\|U_{h}(t)P_{N}u\|_{L_{t}^{q}L_{h}^{r}}^{2}\right)^{\frac{1}{2}}\lesssim \left(\sum_{N\le1}\left\|P_{N}|\nabla_{h}|^{(3-\alpha)(\frac{1}{2}-\frac{1}{r})}u\right\|_{L_{h}^{2}}^{2}\right)^{\frac{1}{2}}
	  	\end{equation*}
	  	\begin{equation*}
	  		\simeq\left\| |\nabla_{h}|^{(3-\alpha)(\frac{1}{2}-\frac{1}{r})}u\right\|_{L_{h}^{2}}.
	  	\end{equation*}
	  \end{proof}
		\begin{coro}\label{coro}
			With the same hypothesis, there exists Strichartz admissible pair $(q_{0},r_{0})$, $p-1<q_{0}$, and uniform $L_{h}^{\infty}$ estimate below
			\begin{equation*}
				\|U_{h}(t)u\|_{L_{t}^{q_{0}}L_{h}^{\infty}}\lesssim \|u\|_{H_{h}^{\frac{\alpha}{2}}}.
			\end{equation*}
		\end{coro}
		\begin{proof}
			We can take $(q_{0},r_{0}):=(\frac{2\alpha}{3-\alpha+\delta},\frac{2\alpha}{2\alpha-3-\delta})$, for some sufficiently small $\delta>0$. Direct calculation shows that such $(q_{0},r_{0})$ satisfies the above requirement. Then we also take   $s:=\frac{\alpha}{2}-(3-\alpha)(\frac{1}{2}-\frac{1}{r_{0}})$, then $s>\frac{3}{r_{0}}$. Consequently, we have
			\begin{equation*}
			\|U_{h}(t)u\|_{L_{t}^{q_{0}}L_{h}^{\infty}}\lesssim \|U_{h}u\|_{L_{t}^{q_{0}}W_{h}^{s,r_{0}}}\lesssim\|u\|_{H_{h}^{\frac{\alpha}{2}}}.
			\end{equation*}
			The first inequality comes from Sobolev embedding (see e.g. \cite{5} or \cite{7}), and the second inequality is derived from Theorem \ref{uniform}.
		\end{proof}
		Finally, We need another and the last lemma before proving our Theorem \ref{main}.
		\begin{lemma}\label{local}
			Still assuming the same hypothesis and $(q_{0},r_{0})$ is given by Corollary \ref{coro}, the solution $u,u_{h}$ for equations (\ref{CFNLS}) and (\ref{FNLS}) are well-posed on time interval $[0,T]$, satisfying
			\begin{equation*}
				\|u\|_{L^{\infty}([0,T]; H^{\frac{\alpha}{2}}(\mathbb{R}^{d}))}+\|u\|_{L^{q_{0}}\left([0,T];L^{\infty}(\mathbb{R}^{d})\right)}\lesssim \|u_{0}\|_{H^{\frac{\alpha}{2}}(\mathbb{R}^{d})},
			\end{equation*}
			\begin{equation*}
				\|u_{h}\|_{L^{\infty}([0,T]; H_{h}^{\frac{\alpha}{2}})}+\|u_{h}\|_{L^{q_{0}}\left([0,T];L_{h}^{\infty}\right)}\lesssim \|u_{0,h}\|_{H_{h}^{\frac{\alpha}{2}}}.
			\end{equation*}
			Furthermore, we have $T\simeq \|u_{0}\|_{H^{\frac{\alpha}{2}}(\mathbb{R}^{d})}^{-\frac{1}{\frac{1}{p-1}-\frac{1}{q_{0}}}}$.
		\end{lemma}
		\begin{proof}
			The above estimates are from classical contraction mapping theorem, which can be referred to \cite{4} and \cite{8}. 
		\end{proof}
		Now we follow the same procedure in \cite{1} or \cite{8}, and give the proof of our main Theorem \ref{main}.
		\begin{proof}[Proof of Theorem \ref{main}]
		    
		  \par From the Lemma \ref{wellposedness} and Lemma \ref{local}, we have unique $u\in C([0,T];H^{\frac{\alpha}{2}}(\mathbb{R}^{d}))$ to equation (\ref{CFNLS}) and unique $u\in C^{1}([0,T];L_{h}^{2})$ to equation (\ref{FNLS}). Furthermore, they satisfy
		  \begin{equation*}
		  	u(t)=U(t)u_{0}-i\mu\int_{0}^{t}U(t-s)\left(|u(s)|^{p-1}u(s)\right) ds,
		  \end{equation*} 
		  \begin{equation*}
		  	u_{h}(t)=U_{h}(t)u_{0,h}-i\mu\int_{0}^{t}U_{h}(t-s)\left(|u_{h}(s)|^{p-1}u_{h}(s)\right) ds.
		  \end{equation*}
		  Then we have 
		  \begin{equation*}
		  p_{h}u_{h}(t)-u(t)=(I)+(II)+(III)+(IV),
		  \end{equation*}
		  where
		  \begin{equation*}
		  	(I)=p_{h}U_{h}(t)u_{0,h}-U(t),
		  \end{equation*}
		  \begin{equation*}
		  	(II)=-i\mu\int_{0}^{t}\left(p_{h}U_{h}(t-s)-U(t-s)p_{h}\right)\left(|u_{h}(s)|^{p-1}u_{h}(s)\right)ds,
		  \end{equation*}
		  \begin{equation*}
		  	(III)=-i\mu \int_{0}^{t}U(t-s)\left(p_{h}(|u_{h}(s)|^{p-1}u_{h}(s))-|p_{h}u_{h}(s)|^{p-1}p_{h}u_{h}(s)\right) ds,
		  \end{equation*}
		  \begin{equation*}
		  	(IV)=-i\mu\int_{0}^{t}U(t-s)\left(|p_{h}u_{h}(s)|^{p-1}p_{h}u_{h}(s)-|u(s)|^{p-1}u(s)\right)ds.
		  \end{equation*}
		  Following the proof of Theorem 1.1 in \cite{1}, we can conclude that
		  \begin{equation*}
		  	\|(I)\|_{L^{2}(\mathbb{R}^{d})}\lesssim h^{\frac{\alpha}{2+\alpha}}\langle t\rangle\|u_{0}\|_{H^{\frac{\alpha}{2}}(\mathbb{R}^{d})},
		  \end{equation*}
		  \begin{equation*}
		  	\|(II)\|_{L^{2}{(\mathbb{R}^{d})}}, \|(III)\|_{L^{2}{(\mathbb{R}^{d})}}\lesssim h^{\frac{\alpha}{2+\alpha}}\langle t\rangle^{2}\|u_{0}\|_{H^{\frac{\alpha}{2}}(\mathbb{R}^{d})}^{p},
		  \end{equation*}
		  \begin{equation*}
		  	\|(IV)\|_{L^{2}{(\mathbb{R}^{d})}}\lesssim \int_{0}^{t}\left(\|u_{h}(s)\|_{L_{h}^{\infty}}+\|u(s)\|_{L^{\infty}(\mathbb{R}^{d})}\right)^{p-1}\|p_{h}u_{h}(s)-u(s)\|_{L^{2}(\mathbb{R}^{d})}ds.
		  \end{equation*}
		  Then the above estimates lead to 
		  \begin{equation*}
		  	\|p_{h}u_{h}(t)-u(t)\|_{L^{2}(\mathbb{R}^{d})}\lesssim h^{\frac{\alpha}{2+\alpha}}\langle t\rangle^{2}\left(\|u_{0}\|_{H^{\frac{\alpha}{2}}(\mathbb{R}^{d})}+\|u_{0}\|_{H^{\frac{\alpha}{2}}(\mathbb{R}^{d})}^{p}\right)
		  \end{equation*}
		  \begin{equation*}
		  	+\int_{0}^{t}\left(\|u_{h}(s)\|_{L_{h}^{\infty}}+\|u(s)\|_{L^{\infty}(\mathbb{R}^{d})}\right)^{p-1}\|p_{h}u_{h}(s)-u(s)\|_{L^{2}(\mathbb{R}^{d})}ds.
		  \end{equation*}
		  Applying Gronwall inequality, we get our desired result.
		\end{proof}
		\begin{rem}
			The above procedure is actually suitable for all dimensions. However, when dimension increases, the derivation of uniform Strichartz estimate (Theorem \ref{uniform}) or Lemma \ref{3.4} gets more and more complicated, which is the major obstacle for extending continuum limit result of (FNLS) to higher dimensions.
		\end{rem}
		\begin{rem}
			Our proof in $H^{\frac{\alpha}{2}}$- regime will fail if the dimension $d\ge 4$. In fact, Theorem 1.1 in \cite{4}, which ensures the local well-posedness of (FNLS) (\ref{CFNLS}) in the $H^{s}$-regime, always requires the conditions that $(p\ge3, s>s_{c})$ or $(p>3,s=s_{c})$. Thus, we should require that 
			\begin{equation*}
				\dfrac{d(p-1)}{p+1}<\alpha<2.
			\end{equation*}
			However, $\frac{d(p-1)}{p+1}\ge2$ and the above range can't be satisfied.
		\end{rem}		
		\section{Proof of Lemma \ref{3.4}}
		In order to prove Lemma \ref{3.4}, it suffices to derive $L_{h}^{\infty}$-estimate of the convolution kernel of operator $U_{h}P_{N}$. Precisely, we have
		\begin{equation*}
			\|U_{h}P_{N}u\|_{L_{h}^{\infty}}=\|K_{t,N.h}\ast u\|_{L_{h}^{\infty}}\le \|K_{t,N,h}\|_{L_{h}^{\infty}}\|u\|_{L_{h}^{1}},
		\end{equation*}
		where convolution kernel $K_{t,N,h}$ is given by
		\begin{equation*}
			K_{t,N,h}=\frac{1}{(2\pi)^{d}}\int_{\frac{\mathbb{T}^{d}}{h}}e^{i\left\lbrace x\cdot\xi-t\left(\frac{4}{h^{2}}\sum_{i=1}^{d}\sin^{2}(\frac{h\xi_{i}}{2})\right)^{\frac{\alpha}{2}}\right\rbrace}\eta(\frac{h\xi}{N})d\xi.
		\end{equation*}
		Change the variable $\xi:=\frac{\xi}{h}$ and denote $\tau:=\frac{2^{\alpha}t}{h^{\alpha}}$, $v:=\frac{x}{h\tau}$, we can further transform it to the following simplified expression
		\begin{equation}\label{simp}
				K_{t,N,h}=\frac{1}{(2\pi h)^{d}}\int_{\mathbb{T}^{d}}e^{i\tau(v\cdot\xi-\omega(\xi))}\eta(\frac{\xi}{N})d\xi=\frac{1}{(2\pi h)^{d}}\int_{\mathbb{T}^{d}}e^{i\tau \Phi_{v}(\xi)}\eta(\frac{\xi}{N})d\xi,
		\end{equation}
		where $\omega(\xi):=\left(\sum_{i=1}^{d}\sin^{2}(\frac{\xi_{i}}{2})\right)^{\frac{\alpha}{2}}$, $\Phi_{v}(\xi):=v\cdot \xi-\omega(\xi)$.
		
		For convenience, we introduce the conventional symbol for oscillatory integral
		\begin{equation}\label{p}
			J_{\Phi_{v},\zeta}(\tau):=\int_{\mathbb{R}^{d}}e^{i\tau\Phi_{v}(\xi)}\zeta(\xi)d\xi.
		\end{equation}
		Thus, the whole estimate reduces to the estimate of oscillatory integral 
		\begin{equation}\label{deal}
		        J_{\Phi_{v}, \eta(\frac{\cdot}{N})}:=\int_{\mathbb{R}^{d}}e^{i\tau\Phi_{v}(\xi)}\eta(\frac{\xi}{N})d\xi.
		\end{equation}
		Precisely, we just need to show $\sup_{v\in\mathbb{R}^{3}}|J_{\Phi, \eta(\frac{\cdot}{N})}|\lesssim N^{3-\alpha}\tau^{-1}$.
		\begin{rem}
		 To transform integral (\ref{simp}), which is on $\mathbb{T}^{d}$, into an $\mathbb{R}^{d}$-integral (\ref{deal}), we can apply the following trick.
		 We first choose a cutoff $\chi\in C_{c}^{\infty}(\mathbb{R}^{d})$, which has the support in $(-2\pi,2\pi)^{d}$ and is non-vanishing on $\mathbb{T}^{d}$. Moreover, we can require such a cutoff satisfies 
		 \begin{equation}\label{q}
		 	\sum_{x\in \Z^{d}}\chi(\xi+2\pi x)=1, \quad \forall \xi\in \mathbb{R}^{d}. 
		 \end{equation}
		 Then periodicly extending $\eta$ and applying the identity (\ref{q}) to integral (\ref{simp}), we obtain
		 \begin{equation*}
		 	\int_{\mathbb{T}^{d}}e^{i\tau \Phi_{v}(\xi)}\eta(\frac{\xi}{N})d\xi=\sum_{x\in \Z^{d}}\int_{\mathbb{T}^{d}}e^{i\tau \Phi_{v}(\xi)}\eta(\frac{\xi}{N})\chi(\xi+2\pi x)d\xi=\int_{\mathbb{R}^{d}}e^{i\tau \Phi_{v}(\xi)}\eta(\frac{\xi}{N})\chi(\xi)d\xi,
		 \end{equation*}
		 which successfully reduces to the $\mathbb{R}^{d}$-integral. Although the integrand has changed, the whole proof below will still work for new integrand. For simplicity, we still regard $\eta(\frac{\cdot}{N})$ as the integrand on $\mathbb{R}^{d}$.
		\end{rem}
		\begin{rem}
			For the convenience of next calculation and without loss of generality, we consider $\omega(\xi)=\left(\sum_{i=1}^{d}2-2\cos(\xi_{i})\right)^{\frac{\alpha}{2}}$, instead of $\left(\sum_{i=1}^{d}\sin^{2}(\frac{\xi_{i}}{2})\right)^{\frac{\alpha}{2}}$.
		\end{rem}
		To estimate the oscillatory integral, we first calculate the derivatives of $\omega(\xi)$.
		\begin{itemize}
			\item \begin{equation*}
				\nabla\omega(\xi)=\frac{\alpha}{\omega(\xi)^{\frac{2-\alpha}{\alpha}}}\left(\sin(\xi_{1}),\cdots,\sin(\xi_{d})\right);
			\end{equation*}
			\item 
			\begin{equation*}
			   Hess_{\xi}\omega:=\left(\omega_{ij}(\xi)\right)_{i,j=1}^{d}=\left(\dfrac{(\alpha-2)\sin(\xi_{i})\sin(\xi_{j})+\delta_{ij}\cos(\xi_{i})\omega(\xi)^{\frac{2}{\alpha}}}{\omega(\xi)^{\frac{4}{\alpha}-1}}\right)_{i,j=1}^{d}.
			\end{equation*}
			
		\end{itemize}
		We call $\xi\in\mathbb{T}^{d}-\lbrace0\rbrace$ is a critical point, if $v=\nabla\omega(\xi)$; $\xi$ is a degenerate point, if $\det Hess_{\xi}\omega(\xi)=0$. And for convenience, we always assume that $\xi\in[0,\pi]^{d}$.
		
		Then we state several primary properties of oscillatory integral (\ref{deal}).
		\begin{thm}
			For $v\in \mathbb{R}^{d}$, we denote $\Omega_{v}$ as
			\begin{equation*}
				\Omega_{v}:=\left\lbrace\xi\in\mathbb{T}^{d}-\lbrace0\rbrace\Big|v=\nabla\omega(\xi)\right\rbrace,
			\end{equation*}
			then we have the following properties of $\Omega_{v}:$
			\begin{itemize}
				\item $\sup_{v\in \mathbb{R}^{d}}\#(\Omega_{v})<\infty$, which implies for any $v$, the number of critical points is uniformly bounded.
				\item For $|v|\gg 1$, $\Omega_{v}=\emptyset$, which implies we only need to deal with $v$ in a bounded set.
			\end{itemize}
		\end{thm}
		\begin{proof}
			For the first statement, we notice that if $v=0$, then obviously $\Omega_{v}=\emptyset$, and the statement holds. Thus, we suppose $v\ne 0$ and denote  $v=(v_{1},\cdots,v_{d})$, $x_{i}:=\sin(\xi_{i})$, $i=1,\cdots,d$. Then we have 
			\begin{equation*}
				 k=\left(\sum_{j=1}^{d}2-2\sqrt{1-x_{j}^{2}}\right)^{\frac{2-\alpha}{2}},\; x_{i}=kv_{i}, \; i=1,\cdots,d.
			\end{equation*}
			Let $y:=k^{2}, a_{i}:=v_{i}^{2}$, the problem reduces to show for all $a_{i}\ge 0$, with at least one $a_{i}>0$, the number of roots in the following equation is uniformly bounded
			\begin{equation}\label{equation}
				y^{\frac{1}{2-\alpha}}=\sum_{j=1}^{d}2-2\sqrt{1-a_{i}y}.
			\end{equation} 
			If not, for any $m\in\mathbb{N}$, there is $\lbrace a_{i}^{(m)}\rbrace_{i=1}^{d}$, s.t. according equation (\ref{equation}) has at least $m$ roots. Using Rolle's theorem, we see that for any $k, m\in\mathbb{N} $, there is $\lbrace a_{i}^{(m,k)}\rbrace_{i=1}^{d}$, s.t. following equation has at least $m$ roots
			\begin{equation}\label{y}
				\frac{d^{k}}{dy^{k}}\left(y^{\frac{1}{2-\alpha}}\right)=\frac{d^{k}}{dy^{k}}\left(\sum_{j=1}^{d}2-2\sqrt{1-a_{j}^{(m,k)}y}\right).
			\end{equation}
			Simple calculation shows that every order derivatives of $-\sqrt{1-a_{i}y}, a_{i}>0$ are positive. Next we take $k$ s.t. $k-1\le\frac{1}{2-\alpha}< k$, then the LHS is always non-positive and the RHS is always positive. Thus the equation (\ref{y}) has at most one root, which leads to a contradiction. This finishes the first statement.
			
			For the second statement, we just need to notice that $|\nabla\omega(\xi)|$ is bounded.
		\end{proof}
		\begin{thm}
			$\xi$ is a degenerate point if and only if
			\begin{equation}\label{i}
				\prod_{i=1}^{d}\cos(\xi_{i})-\frac{2-\alpha}{\omega(\xi)^{\frac{2}{\alpha}}}\sum_{i=1}^{d}\left(\sin^{2}(\xi_{i})\prod_{j\ne i}\cos(\xi_{j})\right)=0,
		    \end{equation}
		    or equivalently, fells into one of the situations below
		     \begin{itemize}
		    		\item $\Gamma_{1}:=\left\lbrace \xi\in\mathbb{T}^{d}-\lbrace0\rbrace \Big| 2d=\sum_{i=1}^{d}(2-\alpha)\sec(\xi_{i})+\alpha\cos(\xi_{i})\right\rbrace;$
		    		\item $\Gamma_{k}:=\left\lbrace \xi\in\mathbb{T}^{d}-\lbrace0\rbrace \Big| \xi \; has \; exactly \; k \; components \; equal \; to \; \frac{\pi}{2} \right\rbrace$, \quad $k=2,\cdots,d.$
		     \end{itemize}
		     Moreover we have 
		     \begin{equation}\label{o}
		     	\min_{\xi\in\Gamma_{1}}|\xi|:=r_{\alpha}>0.
		     \end{equation}
		 \end{thm}
		 \begin{proof}
		 	From the calculation of Hessian matrix, we see that
		 	\begin{equation*}
		 		Hess_{\xi}\omega(\xi)=\frac{1}{\omega(\xi)^{\frac{2}{\alpha}-1}}
		 		\left(\begin{bmatrix}
		 			\cos(\xi_{1}) & &  \\
		 			 
		 			               &\ddots& \\
		 			              &&\cos(\xi_{d}) \\
		 			
		 		\end{bmatrix}-
		 		\frac{2-\alpha}{\omega(\xi)^{\frac{2}{\alpha}}}
		 		\begin{bmatrix}
		 		    \sin(\xi_{1})\\
		 		    
		 		    \vdots \\
		 		    \sin(\xi_{d})\\
		 		\end{bmatrix} \cdot
		 		\begin{bmatrix}
		 			\sin(\xi_{1})&\cdots&\sin(\xi_{d})\\
		 		\end{bmatrix}\right).
		 	\end{equation*}
		 	Using knowledge from linear algebra, we see the determinant of matrix in bracket is equal to LHS of (\ref{i}). Thus the classification of all degenerate points comes directly.
		 	
		 	Next, we prove the statement (\ref{o}). It's suffices to show that the maximum of function $(2-\alpha)\sec(y)+\alpha\cos(y)$ on the neighborhood of $0$, can only be obtained at $y=0$, which is very trivial. 
		 \end{proof}
		 Following the notation in Appendix, we first need to derive the uniform estimate of oscillatory integral (\ref{deal}).
		 \begin{lemma}\label{g}
		 	For $d=3$, $\xi^{\ast}\in \mathbb{T}^{3}-\lbrace0\rbrace$, $v_{\xi^{\ast}}=\nabla\omega(\xi^{\ast})$, we have the following estimates
		 	\begin{itemize}
		 		\item (I):  If $\xi^{\ast}$ is non-degenerate, then $M(\Phi_{v_{\xi^{\ast}}},\xi^{\ast})\curlyeqprec(-\frac{3}{2},0);$
		 		\item (II):  If $\xi^{\ast}\in \Gamma_{3}$, then $M(\Phi_{v_{\xi^{\ast}}},\xi^{\ast})\curlyeqprec(-\frac{7}{6},0);$
		 		\item (III):  If $\xi^{\ast}\in \Gamma_{2}$, then $M(\Phi_{v_{\xi^{\ast}}},\xi^{\ast})\curlyeqprec(-\frac{5}{4},0);$
		 		\item (IV):  If $\xi^{\ast}\in \Gamma_{1}$, then $M(\Phi_{v_{\xi^{\ast}}},\xi^{\ast})\curlyeqprec(-\frac{4}{3},0).$
		 	\end{itemize}
		 \end{lemma}
		 \begin{proof}
		 	For (I), we just need to take a linear transform $\xi:=A\zeta+\xi^{\ast}$, where $A$ satisfies
		 	\begin{equation*}
		 		A^{T}\cdot Hess_\xi{\omega(\xi^{\ast})}\cdot A= diag\lbrace c_{1},c_{2},c_{3}\rbrace, \; c_{i}\ne 0, \forall i=1,2,3.	 		
		 	\end{equation*}
		 	Then  we have the Taylor series at $0$ as follows 
		 	\begin{equation*}
		 		\Phi_{v_{\xi^{\ast}}}(A\zeta+\xi^{\ast})=\left(v_{\xi^{\ast}}\cdot\xi^{\ast}-\omega(\xi^{\ast})\right)+\left(c_{1}\zeta_{1}^{2}+c_{2}\zeta_{2}^{2}+c_{3}\zeta_{3}^{2}\right)+O(|\zeta|^{3})
		 	\end{equation*}
		 	\begin{equation*}
		 		:=r_{\ast}+Q(\zeta)+O(|\zeta|^{3}).
		 	\end{equation*}
		 	Take $\gamma:=(\frac{1}{2},\frac{1}{2},\frac{1}{2})$, we see that $Q(\zeta)\in \mathcal{E}_{\gamma,3}$, $O(|\zeta|^{3})\in H_{\gamma,3}$. Next, we can use the Lemma \ref{E} and Lemma \ref{Q} in Appendix, and  derive 
		 	\begin{equation*}
		 		M(\Phi_{v_{\xi^{\ast}}},\xi^{\ast})\curlyeqprec M(c_{1}\zeta_{1}^{2}+c_{2}\zeta_{2}^{2}+c_{3}\zeta_{3}^{2})
		 	\end{equation*}
		 	\begin{equation*}
		 		\curlyeqprec(-\frac{1}{2},0)+(-\frac{1}{2},0)+(-\frac{1}{2},0)\curlyeqprec(-\frac{3}{2},0).
		 	\end{equation*}
		 	
		 	For (II), $\xi^{\ast}$=$(\frac{\pi}{2},\frac{\pi}{2},\frac{\pi}{2})$. Then we consider the linear transform $\xi:=A\zeta+\xi^{\ast}$, where
		 	\begin{equation*}
		 		A=
		 		\begin{bmatrix}
		 			1&0&-1\\
		 			0&1&-1\\
		 			0&0&1\\
		 		\end{bmatrix}.
		 	\end{equation*}
		 	Then we have the Taylor series at $0$ as follows
		 	\begin{equation*}
		 		\Phi_{v_{\xi^{\ast}}}(A\zeta+\xi^{\ast})=\left(v_{\xi^{\ast}}\cdot\xi^{\ast}-\omega(\xi^{\ast})\right)+ (a\zeta_{3}^{2}+b\zeta_{1}^{2}\zeta_{2}+c\zeta_{1}\zeta_{2}^{2})+R(\zeta)
		 	\end{equation*}
		 	\begin{equation*}
		 		:=r_{\ast}+Q(\zeta)+R(\zeta),
		 	\end{equation*}
		 	where  $R(\zeta)$ contains no $\zeta_{1}^{2}, \zeta_{2}^{2}$-terms and $a,b,c\ne0$. Take $\gamma:=(\frac{1}{3},\frac{1}{3},\frac{1}{2})$, we see that $Q(\zeta)\in \mathcal{E}_{\gamma,3}$, $R(\zeta)\in H_{\gamma,3}$. Applying the Lemma \ref{E}, Lemma \ref{Q} and Lemma \ref{-2/3} in Appendix, and derive
		 	\begin{equation*}
		 		M(\Phi_{v_{\xi^{\ast}}},\xi^{\ast})\curlyeqprec M(a\zeta_{3}^{2}+b\zeta_{1}^{2}\zeta_{2}+c\zeta_{1}\zeta_{2}^{2})
		 	\end{equation*}
		 	\begin{equation*}
		 		\curlyeqprec(-\frac{1}{2},0)+M(b\zeta_{1}^{2}\zeta_{2}+c\zeta_{1}\zeta_{2}^{2})\curlyeqprec(-\frac{1}{2},0)+(-\frac{2}{3},0)\curlyeqprec(-\frac{7}{6},0).
		 	\end{equation*}
		 	
		 	For (III), we suppose, without loss of generality, that $\xi^{\ast}=(\frac{\pi}{2},\frac{\pi}{2},\xi_{0})$, with $\xi_{0}\ne \frac{\pi}{2}$.
		 	Then we consider the linear transform $\xi:=A\zeta+\xi^{\ast}$, where
		 	\begin{equation*}
		 		A=
		 		\begin{bmatrix}
		 			1&0&0\\
		 			-1&1&0\\
		 			0&0&1\\
		 		\end{bmatrix}.
		 	\end{equation*}
		 	Then we have the Taylor series at $0$ as follows
		 	\begin{equation*}
		 		\Phi_{v_{\xi^{\ast}}}(A\zeta+\xi^{\ast})=\left(v_{\xi^{\ast}}\cdot\xi^{\ast}-\omega(\xi^{\ast})\right)+ (a\zeta_{2}^{2}+b\zeta_{2}\zeta_{3}+c\zeta_{3}^{2}+d\zeta_{1}^{2}\zeta_{2}+\lambda\zeta_{1}^{4})+R(\zeta)
		 	\end{equation*}
		 	\begin{equation*}
		 		:=r_{\ast}+Q(\zeta)+R(\zeta),
		 	\end{equation*}
		 	where $R(\zeta)$ contains no $\zeta_{1}^{2},\zeta_{1}^{3}$-terms and $d\ne0$, $b^{2}\ne 4ac$ ($\lambda$ is uncertain). Then we can rotate $\zeta_{2}\zeta_{3}$-plane and eliminate $\zeta_{2}\zeta_{3}$-term. Take $\gamma:=(\frac{1}{4},\frac{1}{2},\frac{1}{2})$, we see that $Q(\zeta)\in \mathcal{E}_{\gamma,3}$, $R(\zeta)\in H_{\gamma,3}$. Applying the Lemma \ref{E}, Lemma \ref{Q} and Lemma \ref{-2/3} in Appendix, and derive 
		 	\begin{equation*}
		 		M(\Phi_{v_{\xi^{\ast}}},\xi^{\ast})\curlyeqprec M(a\zeta_{2}^{2}+c\zeta_{3}^{2}+d\zeta_{1}^{2}\zeta_{2}+\lambda\zeta_{1}^{4})
		 	\end{equation*}
		 	\begin{equation*}
		 		\curlyeqprec(-\frac{1}{2},0)+M(a\zeta_{2}^{2}+d\zeta_{1}^{2}\zeta_{2}+\lambda \zeta_{1}^{4})\curlyeqprec(-\frac{1}{2},0)+(-\frac{3}{4},0)\curlyeqprec(-\frac{5}{4},0).
		 	\end{equation*}
		 	
		 	For (IV), we suppose that $\xi^{\ast}=(\xi_{1}^{\ast},\xi_{2}^{\ast},\xi_{3}^{\ast})$. Then we consider the linear transform $\xi:=A\zeta+\xi^{\ast}$, where
		 	\begin{equation*}
		 		A=
		 		\begin{bmatrix}
		 			\tan(\xi_{1}^{\ast})&-\tan(\xi_{2}^{\ast})&-\tan(\xi_{3}^{\ast})\\
		 			\tan(\xi_{2}^{\ast})&\tan(\xi_{1}^{\ast})&0\\
		 			\tan(\xi_{3}^{\ast})&0&\tan(\xi_{1}^{\ast})\\
		 		\end{bmatrix}.
		 	\end{equation*}
		 	Then we have the Taylor series at $0$ as follows
		 	\begin{equation*}
		 		\Phi_{v_{\xi^{\ast}}}(A\zeta+\xi^{\ast})=\left(v_{\xi^{\ast}}\cdot\xi^{\ast}-\omega(\xi^{\ast})\right)+ (a\zeta_{2}^{2}+b\zeta_{2}\zeta_{3}+c\zeta_{3}^{2}+d\zeta_{1}^{3})+R(\zeta)
		 	\end{equation*}
		 	\begin{equation*}
		 		:=r_{\ast}+Q(\zeta)+R(\zeta),
		 	\end{equation*}
		 	where $R(\zeta)$ contains no $\zeta_{1}^{2}$-terms and $d\ne0$, $b^{2}\ne 4ac$. Then we can rotate $\zeta_{2}\zeta_{3}$-plane and eliminate $\zeta_{2}\zeta_{3}$-term. Take $\gamma:=(\frac{1}{3},\frac{1}{2},\frac{1}{2})$, we see that $Q(\zeta)\in \mathcal{E}_{\gamma,3}$, $R(\zeta)\in H_{\gamma,3}$. Applying the Lemma \ref{E}, Lemma \ref{Q} and Lemma \ref{-2/3} in Appendix, and derive 
		 		\begin{equation*}
		 		M(\Phi_{v_{\xi^{\ast}}},\xi^{\ast})\curlyeqprec M(a\zeta_{2}^{2}+c\zeta_{3}^{2}+d\zeta_{1}^{3})
		 	\end{equation*}
		 	\begin{equation*}
		 		\curlyeqprec(-\frac{1}{2},0)+(-\frac{1}{2},0)+(-\frac{1}{3},0)\curlyeqprec(-\frac{4}{3},0).
		 	\end{equation*}
		 \end{proof}
		 Now we are ready to prove the key Lemma \ref{3.4}. The proof follows similar spirit of the proof in \cite{8}.
		 \begin{proof}[Proof of Lemma \ref{3.4}]
		 	Notice that, if we change the variable $\xi:=N\xi$, then we have
		 	\begin{equation}\label{interpolaion}
		 		\sup_{v\in \mathbb{R}^{3}}\left|J_{\Phi_{v},\eta(\frac{\cdot}{N})}(\tau)\right|=N^{3}\sup_{v\in \mathbb{R}^{3}}|\int_{\mathbb{R}^{3}}e^{i\left(x\cdot\xi-\omega(N\xi)\right)}\eta(\xi)d\xi|\lesssim N^3.
		 	\end{equation}
		 	This trivial inequality will be used for interpolation purpose and we can suppose that
		 	 $\tau\ge1$.
		 	
		 	Next, we choose a dyadic number $N_{\alpha}\ll r_{\alpha},1 $ and $R\gg 1$, then we consider the following cases.

		 	\noindent
		 	\textbf{Case 1: $1\ge N\ge N_{\alpha}.$ }
		 	
		 	For $\xi\in\mathbb{T}^{3}-B(0,\frac{r_{\alpha}}{2})$, we can apply uniform estimate of oscillatory integral in Lemma \ref{g}. To be more specific, we can take small neighborhoods $\xi \in \Omega_{\xi}$ and $v_{\xi}=\nabla\omega(\xi)\in V_{\xi}$, s.t. there exist constant $C_{\xi}>0$, $M_{\xi}\in\mathbb{N}$, and following uniform estimate
		 	\begin{equation}\label{w}
		 		\left|J_{\Phi_{v},\zeta}(\tau)\right|\le C\left(1+|\tau|\right)^{\beta}\|\zeta\|_{C^{M}(\Omega_{\xi})}, \quad \forall v\in V_{\xi}, \forall \zeta \in C_{c}^{\infty}(\Omega_{\xi}),
		 	\end{equation}
		 	where $\beta\in \lbrace -\frac{7}{6}, -\frac{5}{4}, -\frac{4}{3}, -\frac{3}{2}\rbrace$, depends on the type of $\xi$. Further more, we can suppose that $\Omega_{\xi}$ is small enough that 
		 	\begin{equation*}
		 		\inf_{v\in V_{\xi}^{c}, \xi\in\Omega_{\xi}}\left|v-\nabla\omega(\xi)\right|:=c_{\xi}>0.
		 	\end{equation*}
		 	Then for $v\in V_{\xi}^{c}$, we can repeatedly use integration by part and yield similar estimate as follows
		 	\begin{equation}\label{e}
		 		\left|J_{\Phi_{v},\zeta}(\tau)\right|\le C\left(1+|\tau|\right)^{\beta}\|\zeta\|_{C^{M}(\Omega_{\xi})}, \quad \forall v\in V_{\xi}^{c}, \forall \zeta \in C_{c}^{\infty}(\Omega_{\xi}).
		 	\end{equation}

		 	Since $\left\lbrace\Omega_{\xi}\right\rbrace_{\xi}$ forms a open cover of $\mathbb{T}^{3}-B(0,\frac{r_{\alpha}}{2})$, we can take a finite sub-cover $\left\lbrace\Omega_{\xi_{i}}\right\rbrace_{i=1}^{n_{0}}$, for some $n_{0}\in \mathbb{N}$. Suppose that $\left\lbrace \phi_{i}\right\rbrace_{i=1}^{n_{0}}$ is the partition of unity, which is subordinated to the finite sub-cover $\left\lbrace\Omega_{\xi_{i}}\right\rbrace_{i=1}^{n_{0}}$. Then we have $\supp\phi_{i}\subseteq\Omega_{\xi_{i}}$ and can divide the oscillatory integral (\ref{deal}) as follows
		 	\begin{equation}
		 		J_{\Phi_{v},\eta\left(\frac{\cdot}{N}\right)}(\tau)=\sum_{i=1}^{n_{0}}J_{\Phi_{v},\eta\left(\frac{\cdot}{N}\right)\phi_{i}(\cdot)}(\tau).
		 	\end{equation}
		 	And we denote that $C_{0}:=\max_{i}\lbrace C_{\xi_{i}}\rbrace$ and $M_{0}:=\max_{i}\lbrace M_{\xi_{i}}\rbrace$. Direct calculation also shows that 
		 	\begin{equation*}
		 		\|\eta(\frac{\cdot}{N})\phi_{i}(\cdot)\|_{C^{M_{0}}}\lesssim N^{-M_{0}}.
		 	\end{equation*}
		 	From the above estimates (\ref{w}) and (\ref{e}), we obtain that 
		 	\begin{equation}\label{r}
		 	\sup_{v\in \mathbb{R}^{3}}\left|J_{\Phi_{v},\eta(\frac{\cdot}{N})}(\tau)\right|\le C_{0}\left(\sum_{i=1}^{n_{0}}\|\eta(\frac{\cdot}{N})\phi_{i}(\cdot)\|_{C^{M_{0}}}\right)(1+|\tau|)^{-\frac{7}{6}}\lesssim N^{3-\frac{7}{6}\alpha}\tau^{-\frac{7}{6}}.
		 	\end{equation}
		 	The second inequality comes from $N^{-M_{0}}\lesssim N^{3-\frac{7}{6}\alpha}$, as $1\ge N\ge N_{\alpha}$. Then the interpolation between estimate (\ref{interpolaion}) and estimate (\ref{r}) yields
		 	\begin{equation*}
		 		\sup_{v\in \mathbb{R}^{3}}\left|J_{\Phi_{v},\eta(\frac{\cdot}{N})}(\tau)\right|\lesssim \left(N^{3-\frac{7}{6}\alpha}\tau^{-\frac{7}{6}}\right)^{\frac{6}{7}}\cdot\left(N^{3}\right)^{\frac{1}{7}}=N^{3-\alpha}\tau^{-1}.
		 	\end{equation*}
		 	\noindent
		 	\textbf{Case 2: $N< N_{\alpha}$ and $|x|\le R.$ }
		 	
		 	Recall that $v=\frac{x}{h\tau}$, then we have 
		 	\begin{equation*}
		 		\sup_{v\in\mathbb{R}^{s}}\left|J_{\Phi_{v},\eta(\frac{\cdot}{N})}(\tau)\right|=N^{3}\sup_{x\in\mathbb{R}^{3}}\left|\int_{\mathbb{R}^{3}}e^{i(x\cdot\xi-\tau\omega(N\xi))}\eta(\xi)d\xi\right|.
		 	\end{equation*}
		 	To connect the discrete case and the continuous case, we can introduce the following change of variables, which has appeared in \cite{9}.
		 	\begin{equation}\label{variable}
		 		z_{i}:=\frac{2}{N}\sin\left(\frac{N\xi_{i}}{2}\right), \; \xi_{i}=\frac{2}{N}\sin^{-1}\left(\frac{Nz_{i}}{2}\right), \;i=1,2,3,
		 	\end{equation}
		 	with the according Jacobi term 
		 	\begin{equation*}
		 		J_{c}(z)=\prod_{i=1}^{3}\left(1-\left(\frac{Nz_{i}}{2}\right)^{2}\right)^{-\frac{1}{2}}.
		 	\end{equation*}
		 	Then we have 
		 	\begin{equation*}
		 		N^{3}\int_{\mathbb{R}^{3}}e^{i(x\cdot\xi-\tau\omega(N\xi))}\eta(\xi)d\xi=
		 		N^{3}\int_{\mathbb{R}^{3}}e^{i(x\cdot\xi(z)-\tau N^{\alpha}\rho^{\alpha})}\eta(\xi(z))J_{c}(z)dz
		 	\end{equation*}
		 	\begin{equation}\label{I}
		 		=N^{3}\int_{0}^{\infty}e^{-i\tau N^{\alpha}\rho^{\alpha}}G(\rho,x,N)\rho^{2}d\rho:=(I),
		 	\end{equation}
		 	where $(r,\theta)=(r,\theta_{1},\theta_{2})$ and $(\rho,\phi)=(\rho,\phi_{1},\phi_{2})$ are the polar coordinates for $x=(x_{1},x_{2},x_{3})$ and $z=(z_{1},z_{2},z_{3})$, with
		 	\begin{equation*}
		 		G(\rho,x,N):=G(\rho)=\int_{\mathbb{S}^{2}}e^{ix\cdot\xi(z)}\eta(\xi(z))J_{c}(z)d\phi(z)
		 	\end{equation*}
		 	\begin{equation*}
		 		\int_{0}^{2\pi}\int_{0}^{\pi}e^{i\lambda\Phi_{G}(\phi)}\eta(\xi(z))J_{c}(z)d\phi(z)\sin^{2}(\phi_{1})d\phi_{1}d\phi_{2},
		 	\end{equation*}
		 	where we denote $\lambda=\rho r$, and 
		 	\begin{equation*}
		 		\Phi_{G}(\phi)=\frac{2}{\rho N} \sin(\theta_{1})\cos(\theta_{2})\sin^{-1}\left(\frac{N\rho\sin(\phi_{1})\cos(\phi_{2})}{2}\right)
		 	\end{equation*}
		 	\begin{equation}\label{long}
		 		+\frac{2}{\rho N}\sin(\theta_{1})\sin(\theta_{2})\sin^{-1}\left(\frac{N\rho\sin(\phi_{1})\sin(\phi_{2})}{2}\right)+\frac{2}{\rho N}\cos(\theta_{1})\sin^{-1}\left(\frac{N\rho\cos(\phi_{1})}{2}\right).
		 	\end{equation}
		 	By symmetry, we only consider the case $x_{2}\ge x_{1}\ge x_{3}\ge0$, which lead to the range $\theta_{1}, \theta_{2}\in[\frac{\pi}{4},\frac{\pi}{2}]$. From the change of variables (\ref{variable}), we see $z_{i}\approx\xi_{i}$, $\rho\approx|\xi|$, $i=1,2,3$. Then we have 
		 	\begin{equation*}
		 		\supp\eta\subseteq\left\lbrace|\xi|\in [\dfrac{\pi}{2},2\pi]\right\rbrace\Rightarrow\rho\in[\frac{\pi}{4},4\pi],
		 	\end{equation*}
		 	i.e. the integrand in (\ref{I}) is supported on $\rho\in[\frac{\pi}{4},4\pi]$.
		 	
		 	Notice that, integration by part yields
		 	\begin{equation}\label{h}
		 		(I)=\frac{N^{3-\alpha}}{i\alpha\tau}\int_{0}^{\infty}e^{-i\tau N^{\alpha}\rho^{\alpha}}\partial_{\rho}\left(G(\rho)\rho^{3-\alpha}\right)d\rho.
		 	\end{equation}
		 	Therefore, if we can show $\partial_{\rho}\left(G(\rho)\rho^{3-\alpha}\right)$ is bounded, uniformly in $N$, then we finish the proof. Since $\rho\in[\frac{\pi}{4},2\pi]$, we see that
		 	\begin{equation*}
		 		\left|G(\rho)\partial_{\rho}\left(\rho^{3-\alpha}\right)\right|\lesssim 1.
		 	\end{equation*}
		 	Direct calculation shows that
		 	\begin{equation*}
		 		\sup_{N\le N_{\alpha}}\left|\partial_{\rho}\left(e^{i\lambda\Phi_{G}(\phi)}\right)\right|\lesssim r\le R, \; \sup_{N\le N_{\alpha}}\left| \partial_{\rho}\eta\right|,\; \sup_{N\le N_{\alpha}}\left| \partial_{\rho}J_{c}\right|\lesssim 1.
		 	\end{equation*}
		 	Then for fixed $R$, we derive the estimate $|(I)|\lesssim N^{3-\alpha}\tau^{-1}$.

		 	\noindent
		 	\textbf{Case 3: $N< N_{\alpha}$ and $|x|> R.$ }
		 	
		 	Then we analyze the property of $G(\rho)$, especially its critical points. Direct calculation shows that
		 	\begin{itemize}
		 		\item 
		 		\begin{equation*}
		 			\partial_{\phi_{1}}\Phi_{G}=\dfrac{\sin(\theta_{1})\cos(\theta_{2})\cos(\phi_{1})\cos(\phi_{2})}{\sqrt{1-\left(\dfrac{N\rho\sin(\phi_{1})\cos(\phi_{2})}{2}\right)^{2}}}
		 		\end{equation*}
		 		\begin{equation*}
		 			+\dfrac{\sin(\theta_{1})\sin(\theta_{2})\cos(\phi_{1})\sin(\phi_{2})}{\sqrt{1-\left(\dfrac{N\rho\sin(\phi_{1})\sin(\phi_{2})}{2}\right)^{2}}}-\dfrac{\cos(\theta_{1})\sin(\phi_{1})}{\sqrt{1-\left(\dfrac{N\rho\cos(\phi_{1})}{2}\right)^{2}}};
		 		\end{equation*}
		 		\item 
		 		\begin{equation*}
		 			\partial_{\phi_{2}}\Phi_{G}=-\dfrac{\sin(\theta_{1})\cos(\theta_{2})\sin(\phi_{1})\sin(\phi_{2})}{\sqrt{1-\left(\dfrac{N\rho\sin(\phi_{1})\cos(\phi_{2})}{2}\right)^{2}}}+\dfrac{\sin(\theta_{1})\sin(\theta_{2})\sin(\phi_{1})\cos(\phi_{2})}{\sqrt{1-\left(\dfrac{N\rho\sin(\phi_{1})\sin(\phi_{2})}{2}\right)^{2}}}.
		 		\end{equation*}
		 	\end{itemize}
		 	Then we classify the critical points, i.e. $\nabla_{\phi}\Phi_{G}=(\partial_{\phi_{1}}\Phi_{G},\partial_{\phi_{2}}\Phi_{G})=0$.
		 	\begin{itemize}
		 		\item (First-kind critical points): If $\sin(\phi_{1})=0$, i.e. $\phi_{1}=0$ or $\pi$, then we have
		 		\begin{equation*}
		 			\phi_{2}=\theta_{2}+\frac{\pi}{2} \quad or\quad \theta_{2}+\frac{3\pi}{2}.
		 		\end{equation*} 
		 		\item (Second-kind critical points): If $\sin(\phi_{1})\ne0$, i.e. $\phi_{1}\ne0$ and $\pi$, then we have
		 		\begin{equation}\label{zzz}
		 			\tan(\phi_{2})\cdot g(\rho,\phi_{1},\phi_{2})=\tan(\theta_{2}),
		 		\end{equation}
		 		\begin{equation}\label{z}
		 			\tan(\phi_{1})\cdot\sin(\phi_{2})\cdot h(\rho,\phi_{1},\phi_{2})=\tan(\theta_{1})\cdot \sin(\theta_{2}),
		 		\end{equation}
		 		where 
		 		\begin{equation}\label{zz}
		 			g(\rho,\phi_{1},\phi_{2})=\left(\dfrac{1-\left(\dfrac{N\rho\sin(\phi_{1})\sin(\phi_{2})}{2}\right)^{2}}{1-\left(\dfrac{N\rho\sin(\phi_{1})\cos(\phi_{2})}{2}\right)^{2}}\right)^{\dfrac{1}{2}},
		 		\end{equation}
		 		\begin{equation}\label{zzzz}
		 			h(\rho,\phi_{1},\phi_{2})=\left(\dfrac{1-\left(\dfrac{N\rho\sin(\phi_{1})\sin(\phi_{2})}{2}\right)^{2}}{1-\left(\dfrac{N\rho\cos(\phi_{1})}{2}\right)^{2}}\right)^{\dfrac{1}{2}}.
		 		\end{equation}
		 	\end{itemize}
		 	Observe that $N<N_{\alpha}\ll1$, we have $|g(\rho,\phi_{1},\phi_{2})-1|\ll1, |h(\rho,\phi_{1},\phi_{2})-1|\ll1. $ Thus, from the equations (\ref{zzz}), (\ref{z}), we deduce that
		 	\begin{equation*}
		 		\phi_{1}\approx\theta_{1}, \phi_{2}\approx\theta_{2}\quad or \quad  \phi_{1}\approx\pi-\theta_{1}, \phi_{2}\approx\pi+\theta_{2}.
		 	\end{equation*}
		 	In conclusion, for fixed $\theta=(\theta_{1},\theta_{2})$, we have $6$ critical points $\phi^{(k)}:=\left(\phi_{1}^{(k)},\phi_{2}^{(k)}\right)$, $k=1,2,3,4,5,6$. Furthermore, the distances among them have a uniform positive lower bound.
		 	
		 	Next, we claim that the critical points are all non-degenerate and the determinants of their Hessian matrices also have a uniform positive lower bound. In fact, from the condition $N<N_{\alpha}\ll1$ and calculation, we have
		 	\begin{itemize}
		 		\item
		 		\begin{equation*}
		 			\partial_{\phi_{1}}^{2}\Phi_{G}\approx -\sin(\theta_{1})\sin(\phi_{1})\cos(\theta_{2}-\phi_{2})-\cos(\theta_{1})\cos(\phi_{1});
		 		\end{equation*}
		 		\item 
		 		\begin{equation*}
		 			\partial_{\phi_{2}}^{2}\Phi_{G}\approx -\sin(\theta_{1})\sin(\phi_{1})\cos(\theta_{2}-\phi_{2});
		 		\end{equation*}
		 		\item 
		 		\begin{equation*}
		 			\partial_{\phi_{1}\phi_{2}}^{2}\Phi_{G}\approx -\sin(\theta_{1})\cos(\phi_{1})\sin(\phi_{2}-\theta_{2}).
		 		\end{equation*}
		   \end{itemize} 
		 	It can be checked that $\det(Hess_{\phi}\Phi_{G}(\phi^{(k)}))\ge c>0$, $k=1,2,3,4,5,6,$ where $c$ is independent of $N$. Then we introduce cutoffs $\chi_{k}\in C_{c}^{\infty}({U_{k}})$, where $U_{k}$ is a small neighborhood of $\phi^{(k)},$ and $\chi_{k}\equiv1$ in a smaller neighborhood of $\phi^{(k)}$, $k=1,2,3,4,5,6.$ And we denote that $\chi_{0}:=1-\sum_{k=1}^{6}\chi_{k}$, with
		 	\begin{equation*}
		 		\widetilde{\chi_{k}}(\phi):=\chi_{k}(\phi)\eta(\xi(z)) J_{c}(z)\sin^{2}({\phi_{1}}), \; k=0,1,2,3,4,5,6.
		 	\end{equation*}
		 	Then we divide $G(\rho)$ into several parts
		 	\begin{equation*}
		 		G(\rho)=\sum_{k=0}^{6}	\int_{0}^{2\pi}\int_{0}^{\pi}e^{i\lambda\Phi_{G}(\phi)}\chi_{k}(\phi_{1},\phi_{2})\eta(\xi(z))J_{c}(z)\sin^{2}(\phi_{1})d\phi_{1}d\phi_{2}
		 	\end{equation*}
		 	\begin{equation*}
		 		=\sum_{k=0}^{6}	\int_{0}^{2\pi}\int_{0}^{\pi}e^{i\lambda\Phi_{G}(\phi)}\widetilde{\chi_{k}}(\phi_{1},\phi_{2})d\phi_{1}d\phi_{2}:=\sum_{k=0}^{6}G_{k}(\rho).
		 	\end{equation*}
		 	For $k=0$, we have $|\nabla_{\phi}\Phi_{G}|>0$ in the support of $\widetilde{\chi_{0}}$. Therefore, integration by part shows the estimate that
		 	\begin{equation*}
		 		|\partial_{\lambda}G_{0}|\lesssim\lambda^{-1}.
		 	\end{equation*} 
		 	For $k=1,2,3,4,5,6$, we have $\det(Hess_{\phi}\Phi_{G})\ne 0$ in the support of $\widetilde{\chi_{k}}$. Then we have the classical asymptotic expression (see e.g. \cite{19} Proposition 6 in Chapter VIII) that
		 	\begin{equation*}
		 		G_{k}(\rho)=\dfrac{2\pi i}{\sqrt{\left|\det(Hess_{\phi}\Phi_{G}) \right|}}e^{i\lambda\Phi_{G}(\phi^{(k)})}\widetilde{\chi_{k}}(\phi^{(k)})\lambda^{-1}+\widetilde{G}_{k}(\lambda), \; \lambda\gg1,
		 	\end{equation*}
		 	where $\left|\partial_{\lambda}^{m}\widetilde{G}_{k}(\lambda)\right|\lesssim_{m}\lambda^{-(m+1)}$, $\forall m\ge0$.
		 	
		 	Since $|\partial_{\rho}G_{0}|=r|\partial_{\lambda}G_{0}|\lesssim \frac{r}{\lambda}=\frac{1}{\rho}$, we can replace $G$ with $G_{0}$ in integral (\ref{h}) and yield desired bound $N^{3-\alpha}\tau^{-1}$. Similarly, $\widetilde{G}_{k}(\lambda)$ will also yield such bound. Thus, it remains to estimate (I) with $G$ replaced by the leading term of $G_{k}$, which can be reduced to 
		 	\begin{equation*}
		 		(II):=N^{3}r^{-1}\int_{0}^{\infty}e^{i\tau\Psi(\rho)}a(\rho)d\rho,
		 	\end{equation*}
		 	where
		 	\begin{equation*}
		 		a(\rho):= \dfrac{\eta(\rho,\phi^{(k)})J_{c}(\rho,\phi^{(k)})\rho}{\sqrt{\left|\det(Hess_{\phi}\Phi_{G}) \right|}},
		 	\end{equation*}
		 	\begin{equation*}
		 		\Psi(\rho):=-N^{\alpha}\rho^{\alpha}+\Phi_{G}(\rho,\phi^{(k)})\frac{r\rho}{\tau}.
		 	\end{equation*}
		 	\noindent
		 	\textbf{Sub-case 3.1: $\frac{r}{\tau}\gg N|\nabla\omega(N\xi)|$ or $\frac{r}{\tau}\ll N|\nabla\omega(N\xi)|$}.
		 	
		 	Direct calculation shows that $|\nabla\omega(N\xi)|\simeq N^{\alpha-1}$. Then under the condition of this sub-case, we have
		 	\begin{equation*}
		 		\left| \nabla_{\xi}\left(\frac{x}{\tau}\cdot\xi-\omega(N\xi)\right)\right|\ge \frac{N}{2}|\nabla\omega(N\xi)|\simeq N^{\alpha}.
		 	\end{equation*}
		 	Thus, using integration by part, we can deduce
		 	\begin{equation*}
		 		N^{3}\left|\int_{\mathbb{R}^{3}}e^{i(x\cdot\xi-\tau\omega(N\xi))}\eta(\xi)d\xi\right|\lesssim N^{3-\alpha}\tau^{-1},
		 	\end{equation*}
		 	which finishes the proof.
		 	
		 	\noindent
		 	\textbf{Sub-case 3.2: $\frac{r}{\tau}\simeq N|\nabla\omega(N\xi)|\simeq N^{\alpha}$ }.

		 	We claim that $\left|\partial_{\rho}^{2}\left(\Phi_{G}(\phi^{(k)})\cdot\rho\right)\right|\lesssim N^2$. Since $\phi^{(k)}=(\phi_{1}^{(k)},\phi_{2}^{(k)})$ also depends on $\rho$, the concrete expression will be extremely long. For simplicity, we ignore the first two terms in the expression (\ref{long}) of $\Phi_{G}$.
		 	\begin{itemize}
		 		\item
		 		 \begin{equation*}
		 			\partial_{\rho}\left(\Phi_{G}(\phi^{(k)})\cdot\rho\right)=\dfrac{\cos(\theta_{1})}{\sqrt{1-\left(\frac{N\rho\cos(\phi_{1}^{(k)})}{2}\right)^{2}}}\cdot\left[\cos(\phi_{1}^{(k)})-\rho\sin(\phi_{1}^{(k)})\partial_{\rho}\phi_{1}^{(k)}\right];
		 		\end{equation*}
		 		\item 
		 		\begin{equation*}
		 			\partial_{\rho}^{2}\left(\Phi_{G}(\phi^{(k)})\cdot\rho\right)=\dfrac{N^{2}\rho}{4}\cdot \dfrac{\cos(\theta_{1})\cos(\phi_{1}^{(k)})}{\left(1-\left(\frac{N\rho\cos(\phi_{1}^{(k)})}{2}\right)^{2}\right)^{\frac{3}{2}}}\cdot\left[\cos(\phi_{1}^{(k)})-\rho\sin(\phi_{1}^{(k)})\partial_{\rho}\phi_{1}^{(k)}\right]^{2}
		 		\end{equation*}
		 		\begin{equation*}
		 			+\dfrac{\cos(\theta_{1})}{\sqrt{1-\left(\frac{N\rho\cos(\phi_{1}^{(k)})}{2}\right)^{2}}}\cdot\left[-2\sin(\phi_{1}^{(k)})\partial_{\rho}\phi_{1}^{(k)}-\rho\cos(\phi_{1}^{(k)})\left(\partial_{\rho}\phi_{1}^{(k)}\right)^{2}-\rho\sin(\phi_{1}^{(k)})\partial_{\rho}^{2}\phi_{1}^{(k)}\right].
		 		\end{equation*}
		 	\end{itemize}
		 	Thus, it suffices to show $\left|\partial_{\rho}\phi_{i}^{(k)}\right|$, $\left|\partial_{\rho}^{2}\phi_{i}^{(k)}\right|\lesssim N^{2}$, $i=1,2.$
		 	
		 	If $\phi^{(k)}=(\phi_{1}^{(k)},\phi_{2}^{(k)})$ is the first-kind critical point, then $\partial_{\rho}\phi_{i}^{(k)},\partial_{\rho}^{2}\phi_{i}^{(k)} \equiv0, i=1,2$. From the above expression of $\partial_{\rho}^{2}\left(\Phi_{G}(\phi^{(k)})\cdot\rho\right)$, we directly deduce that $\left|\partial_{\rho}^{2}\left(\Phi_{G}(\phi^{(k)})\cdot\rho\right)\right|\lesssim N^2$. Thus, we only need to consider the case of second-kind critical points.
		 	
		 	Differentiate $g(\rho,\phi_{1}^{(k)},\phi_{2}^{(k)})$,  $h(\rho,\phi_{1}^{(k)},\phi_{2}^{(k)})$ in equations (\ref{zz}) and (\ref{zzzz}), we can check that $\partial_{\rho}g, \partial_{\rho}h$ are of the form like 
		 	\begin{equation}\label{jj}
		 		a(\rho,\phi_{1}^{(k)},\phi_{2}^{(k)})\partial_{\rho}\phi_{1}^{(k)}+b(\rho,\phi_{1}^{(k)},\phi_{2}^{(k)})\partial_{\rho}\phi_{2}^{(k)}+c(\rho,\phi_{1}^{(k)},\phi_{2}^{(k)}),
		 	\end{equation}
		 	where $|a|,|b|,|c|\lesssim N^{2}$. Differentiate them again, we can derive $\partial_{\rho}^{2}g, \partial_{\rho}^{2}h$, which are of the form like 
		 	\begin{equation*}
		 		\widetilde{a}(\rho,\phi_{1}^{(k)},\phi_{2}^{(k)})\partial_{\rho}^{2}\phi_{1}^{(k)}+\widetilde{b}(\rho,\phi_{1}^{(k)},\phi_{2}^{(k)})\partial_{\rho}^{2}\phi_{2}^{(k)}+\widetilde{c}(\rho,\phi_{1}^{(k)},\phi_{2}^{(k)})\partial_{\rho}\phi_{1}^{(k)}\partial_{\rho}\phi_{2}^{(k)}
		 	\end{equation*}
		 	\begin{equation}\label{kk}
		 		+\widetilde{d}(\rho,\phi_{1}^{(k)},\phi_{2}^{(k)})\partial_{\rho}\phi_{1}^{(k)}+\widetilde{e}(\rho,\phi_{1}^{(k)},\phi_{2}^{(k)})\partial_{\rho}\phi_{2}^{(k)}+\widetilde{f}(\rho,\phi_{1}^{(k)},\phi_{2}^{(k)}),
		 	\end{equation}
		 	where $|\widetilde{a}|,|\widetilde{b}|,|\widetilde{c}|,|\widetilde{d}|,|\widetilde{e}|,|\widetilde{f}|\lesssim N^{2}$.
		 	
		 	Next, we differentiate equations (\ref{z}) and (\ref{zzz}), and obtain
		 	\begin{equation*}
		 		\partial_{\rho}\phi_{2}^{(k)}\cdot g=-\sin(\phi_{2}^{(k)})\cos(\phi_{2}^{(k)})\partial_{\rho}g,
		 	\end{equation*}
		 	\begin{equation}\label{ll}
		 		\partial_{\rho}\phi_{1}^{(k)}\cdot\left(\sin(\phi_{2}^{(k)})h\right)=-\partial_{\rho}\phi_{2}^{(k)}\left(\sin(\phi_{1}^{(k)})\cos(\phi_{1}^{(k)})h\right)-\sin(\phi_{1}^{(k)})\cos(\phi_{1}^{(k)})\sin(\phi_{2}^{(k)})\partial_{\rho}h.
		 	\end{equation}
		 	Notice that $1\lesssim|\sin(\phi_{2}^{(k)})|$, as $\theta_{2}\in [\frac{\pi}{4},\frac{\pi}{2}]$, $\phi_{2}^{(k)}\approx \theta_{2}$ or $\pi+\theta_{2}$, and $g,h\approx 1$. Then we substitute (\ref{jj}) into (\ref{ll}) and solve $\partial_{\rho}\phi_{i}^{(k)}$, which leads to the estimate $\left|\partial_{\rho}\phi_{i}^{(k)}\right|\lesssim N^{2}$.
		 	
		 	To estimate $\left|\partial_{\rho}^{2}\phi_{i}^{(k)}\right|$, we differentiate (\ref{ll}) again, and yields
		 	\begin{equation*}
		 		\partial_{\rho}^{2}\phi_{2}^{(k)}\cdot g=-\partial_{\rho}\phi_{2}^{(k)}\cdot\left(1+\cos(2\phi_{2}^{(k)})\right)\partial_{\rho}g-\sin(\phi_{2}^{(k)})\cos(\phi_{2}^{(k)})\partial_{\rho}^{2}g,
		 	\end{equation*}
		 	\begin{equation*}
		 		\partial_{\rho}^{2}\phi_{1}^{(k)}\cdot\left(\sin(\phi_{2}^{(k)})h\right)=-\partial_{\rho}^{2}\phi_{2}^{(k)}\cdot\left(\frac{1}{2}\sin(2\phi_{1}^{(k)})h\right)-\partial_{\rho}\phi_{1}^{(k)}\partial_{\rho}\phi_{2}^{(k)}\cdot \left(\cos(\phi_{2}^{(k)})+\cos(2\phi_{1}^{(k)})\right)
		 	\end{equation*}
		 	\begin{equation*}
		 		-\partial_{\rho}\phi_{1}^{(k)}\cdot\left(1+\cos(2\phi_{1}^{(k)})\right)\sin(\phi_{2}^{(k)})\partial_{\rho}h-\frac{1}{2}\partial_{\rho}\phi_{2}^{(k)}\left(1+\cos(\phi_{2}^{(k)})\right)\sin(2\phi_{1}^{(k)})\partial_{\rho}h
		 	\end{equation*}
		 	\begin{equation}\label{mm}
		 		-\frac{1}{2}\sin(2\phi_{1}^{(k)})\sin(\phi_{2}^{(k)})\partial_{\rho}^{2}h.
		 	\end{equation}
		 	Combined with the previous estimate $\left|\partial_{\rho}\phi_{i}^{(k)}\right|\lesssim N^{2}$ and (\ref{kk}), we can similarly solve $\partial_{\rho}^{2}\phi_{i}^{(k)}$ and deduce that $\left|\partial_{\rho}^{2}\phi_{i}^{(k)}\right|\lesssim N^{2}$.
		 	
		 	Now, we can estimate (II). Observe that, we have 
		 	\begin{equation*}
		 		\left|\partial_{\rho}^{2}\Psi\right|\ge \alpha(\alpha-1)N^{\alpha}\rho^{\alpha-2}-\left|\partial_{\rho}^{2}\left(\Phi_{G}(\phi^{(k)}\cdot \rho)\right)\right|\frac{r}{\tau}
		 	\end{equation*}
		 	\begin{equation*}
		 		\gtrsim (\alpha-1)N^{\alpha}-N^{2}\frac{r}{\tau}\gtrsim_{\alpha} N^{\alpha}.
		 	\end{equation*}
		 	Thus, from Van der Corput lemma, we can deduce that
		 	\begin{equation*}
		 		|(II)|\lesssim N^{3-\frac{3}{2}\alpha}\tau^{-\frac{3}{2}}\left(\|a\|_{L^{\infty}\left([\frac{\pi}{4},4\pi]\right)}+\|\partial_{\rho}a\|_{L^{1}\left([\frac{\pi}{4},4\pi]\right)}\right).
		 	\end{equation*}
		 	It can be easily check that 
		 	\begin{equation*}
		 		\sup_{N\le N_{\alpha}, \rho\in[\frac{\pi}{4},4\pi]}|a(\rho)|\lesssim 1, \sup_{N\le N_{\alpha}, \rho\in[\frac{\pi}{4},4\pi]} |\partial_{\rho}a(\rho)|\lesssim 1.
		 	\end{equation*}
		 	Therefore, we have completed the whole proof of Lemma \ref{3.4} by using the interpolation. 
		 \end{proof}

		\section{Appendix}
		In this appendix, we will introduce the uniform estimate of oscillatory integral and related  Newton polyhedron, which have been used in the proof of Lemma \ref{3.4}. 
		
		Following the notation in \cite{11} or \cite{12}, we first introduce some concepts.
		\begin{defi}
			For $r,s>0$, the space $\mathcal{H}_{r}(s)$ is defined as follows
			\begin{equation*}
				\mathcal{H}_{r}(s):=\left\lbrace P\Big| P\in \mathcal{O}(B_{\mathbb{C}^{d}}(0,r))\cap C(\overline{B}_{\mathbb{C}^{d}}(0,r)), |P(w)|<s, \forall w\in \overline{B}_{\mathbb{C}^{d}}(0,r) \right\rbrace.
			\end{equation*}
		\end{defi}
		\begin{defi}
			Suppose that $h:\mathbb{R}^{d}\to \mathbb{R}$ is real analytic at $0$. We write 
			\begin{equation*}
				M(h)\curlyeqprec (\beta,p), \; \beta\le 0, p\in \mathbb{N},
			\end{equation*}
			if for $r>0$ sufficiently small, there exist $s>0$, $C>0$ and neighborhood $\Omega\subseteq B_{\mathbb{R}^{d}}(0,r)$ of the origin, s.t.
			\begin{equation*}
				\left|J_{h+P,\zeta}(\tau)\right|\le C(1+|\tau|)^{\beta}\log^{p}(2+|\tau|)\|\zeta\|_{C^{N}(\Omega)}, \; \forall \tau\in \mathbb{R}, \zeta\in C_{c}^{\infty}(\Omega), P\in \mathcal{H}_{r}(s),
			\end{equation*}
			 where $J$ is defined as in (\ref{p}), $N=N(h)\in \mathbb{N}$, with 
			 \begin{equation*}
			 	\|\zeta\|_{C^{N}(\Omega)}:=\sup\left\lbrace|\partial^{\gamma}\zeta(\xi)|\Big|\xi\in\Omega, \gamma\in \mathbb{N}^{d}, |\gamma|\le N\right\rbrace.
			 \end{equation*}
			 We have the following writing convention
			 \begin{itemize}
			 	\item We write $M(h,\xi)\curlyeqprec(\beta,p)$, if 
			 	\begin{equation*}
			 		M(\tau_{\xi}h)\curlyeqprec(\beta,p), \; where \; \tau_{\xi}h(y)=h(y+\xi), \forall y\in\mathbb{R}^{d};
			 	\end{equation*}
			 	\item We write $M(h_{2})\curlyeqprec M(h_{1})+(\beta_{2},p_{2})$, if 
			 	\begin{equation*}
			 		M(h_{1})\curlyeqprec(\beta_{1},p_{1}) \;\; implies \; \; M(h_{2})\curlyeqprec(\beta_{1}+\beta_{2},p_{1}+p_{2}).
			 	\end{equation*}
			 \end{itemize} 
			 
		\end{defi}
		\vspace{10pt}
		Let $\gamma=(\gamma_{1},\cdots, \gamma_{d})\in \mathbb{R}^{d}$, with $\gamma_{i}>0, \forall i=1,\cdots,d$. For any $c>0$, we define the dilation as follows
		\begin{equation*}
			\delta_{c}^{\gamma}(\xi):=\left(c^{\gamma_{1}}\xi_{1},\cdots,c^{\gamma_{d}}\xi_{d}\right), \forall \xi\in \mathbb{R}^{d}.
		\end{equation*}
		\begin{defi}
			A polynomial $f$ on $\mathbb{R}^{d}$ is called $\gamma$-homogeneous of degree $\rho$, if 
			\begin{equation*}
				f\circ\delta_{c}^{\gamma}(\xi)=c^{\rho}f(\xi), \forall \xi\in\mathbb{R}^{d}, c>0.
			\end{equation*}
		\end{defi}
		  Let $\mathcal{E}_{\gamma,d}$ be the set of all $\gamma$-homogeneous polynomials on $\mathbb{R}^{d}$ of degree $1$. $H_{\gamma,d}$ is the set of all functions real-analytic at $0$, with the Taylor's series having the form of $\sum_{\gamma\cdot\alpha>1}a_{\alpha}\xi^{\alpha}$, i.e. the monomial is $\gamma$-homogeneous of degree $>1$.
		  
		  Then we briefly introduce some useful lemmas, which can be seen in \cite{11} or \cite{12}.
		  \begin{lemma}
		  	If $h$ is real analytic at $0$ and $\nabla h(0)\ne 0$, then
		  	\begin{equation*}
		  		M(h)\curlyeqprec (-n,0), \forall n\in\mathbb{N}.		  
		  		\end{equation*}
		  \end{lemma}
		\begin{lemma}\label{E}
			If $h\in \mathcal{E}_{\gamma,d}$ and $P\in H_{\gamma,d}$, then 
			\begin{equation*}
				M(h+P)\curlyeqprec M(h).
			\end{equation*}
		\end{lemma}
		\begin{lemma}\label{Q}
			Let $m,n\ge 1$ and 
			\begin{equation*}
				h_{2}(\xi,y)=h_{1}(\xi)+Q(y), \forall \xi\in\mathbb{R}^{n}, y\in \mathbb{R}^{m},
			\end{equation*}
			where $Q(y)=\sum_{j=1}^{m}c_{j}y_{j}^{2}$, with $c_{j}=\pm1, j=1,\cdot,m$. Then  we have
			\begin{equation*}
				M(h_{2})\curlyeqprec M(h_{1})+(-\frac{m}{2},0).
			\end{equation*}
		\end{lemma}
		Next, we introduce the Newton polyhedron and related theorems, which will be used to derive the uniform estimate for some specific polynomials that appeared in the proof of Lemma \ref{g}. The concepts can be referred to \cite{13,14}
		
		Let $S:\mathbb{R}^{d}\to \mathbb{R}$ be real-analytic at $0$, satisfying
		\begin{equation}\label{a}
			S(0)=0,\quad \nabla S(0)=0.
		\end{equation} 
		Suppose the according Taylor sires at $0$ is 
		\begin{equation*}
			S(\xi)=\sum_{\gamma\in\mathbb{N}^{d}}s_{\gamma}\xi^{\gamma}.
		\end{equation*}
		We also define the Taylor support set $\supp S:=\left\lbrace \gamma\in\mathbb{N}^{d}\Big|s_{\gamma}\ne 0 \right\rbrace$.
		\begin{defi}
			The Newton polyhedron $\mathcal{N}(S)$ of such $S$, is the convex hull of 
			\begin{equation*}
				\bigcup_{\gamma\in \supp(S)}\left(\gamma+\mathbb{R}_{+}^{d}\right), \quad  \mathbb{R}_{+}^{d}:=\left\lbrace \xi\in\mathbb{R}^{d}\Big|\xi_{i}>0, i=1,\cdots,d\right\rbrace.
			\end{equation*}
			If $\mathcal{P}$ is a face of Newton polyhedron $\mathcal{N}(S)$, then we denote $S_{\mathcal{P}}(\xi):=\sum_{\gamma\in \mathcal{P}}s_{\gamma}\xi^{\gamma}$.
		\end{defi}
		\begin{defi}
			S is called $\mathbb{R}-$nondegenerate if for any compact face $\mathcal{P}$,
			\begin{equation*}
				\bigcap_{i=1}^{d}\left\lbrace\xi\in\mathbb{R}^{d}\Big|\partial_{i}S_{\mathcal{P}}(\xi)=0\right\rbrace \subseteq\bigcup_{i=1}^{d}\left\lbrace\xi\in\mathbb{R}^{d}\Big|\xi_{i}=0\right\rbrace ,
			\end{equation*}
			i.e. $\nabla S_{\mathcal{P}}$ is non-vanishing on $(\mathbb{R}-\lbrace0\rbrace)^{d}$.
		\end{defi}
		\begin{defi}
			If $\supp(S)\ne \emptyset$, then the Newton distance $d_{S}$ is defined as
			\begin{equation*}
				d_{S}:=\inf\left\lbrace d>0\Big| (d,\cdots,d)\in \mathcal{N}(S)\right\rbrace.
			\end{equation*}
			The principal face $\pi_{S}$ is the face of minimal dimension that intersects with $\lbrace \xi_{1}=\cdots=\xi_{d}\rbrace$.
		     We also denote $S_{\pi}:=S_{\pi_{S}}$, $k_{S}:=d-\dim_{\mathbb{R}^{d}}(\pi_{S})$, where $\dim_{\mathbb{R}^{d}}(\cdot)$ means affine dimension. 
	\end{defi}
		\begin{defi}
			The height of $S$ is defined as follows
			\begin{equation*}
				h_{S}:=\sup\lbrace d_{S,\xi}\rbrace,
			\end{equation*}
			where the supremum is taken over all local analytic coordinate system $\xi$, preserving the $0$, and $d_{S,\xi}$ is the according Newton distance. 
			
			A coordinate system $\xi_{\ast}$ is called adapted, if $d_{S,\xi^{\ast}}=h_{s}$.
		\end{defi}
		\vspace{7.5pt}
		To recognize if a coordinate system is adapted, there is a very useful theorem for $d=2$.
		\begin{thm}\label{superadapted}
			Let $d=2$, if the principal face $\pi_{S}$ is compact and both the functions $S_{\pi}(1,y)$ and $S_{\pi}(-1,y)$ have no real root of order $\ge d_{S}$ other than possibly $y=0$, then the coordinate system is adapted. 
		\end{thm}
		\begin{proof}
		   See e.g. \cite{15} or \cite{16}.
		\end{proof}
		It's well-known that the oscillatory integral $J_{S,\zeta}(\tau)$ has the asymptotic expansion
		\begin{equation*}
			J_{S,\zeta}(\tau)\approx \sum_{\beta}\sum_{\rho=1}^{d-1}c_{\beta,\rho,\zeta}\tau^{\beta}\log^{\rho}(\tau),
		\end{equation*}
		where $\beta$ runs through finitely many arithmetic progressions of negative rational numbers. More details about this asymptotic expansion can be found in \cite{14,18}. Let $(\beta_{S},\rho_{S})$ be the maximum over all pair $(\beta,\rho)$, s.t  for any neighborhood $\Omega$ of the $0$, there exists $\zeta\in C_{c}^{\infty}(\Omega)$ and according $c_{\beta_{S},\rho_{S},\zeta}\ne 0$. And we usually call such $\beta_{S}$ as the oscillatory index of $S$ and $\rho_{S}$ as the multiplicity of $S$.
		
		Then we have following theorem for $d=2$, that connects adapted coordinate system and $(\beta_{S},\rho_{S})$.
		\begin{thm}\label{uniform estimate}
			If $S$ satisfies (\ref{a}), then there exists coordinate system that is adapted to $S$. Furthermore, we have 
			\begin{equation*}
				M(S)\curlyeqprec(\beta_{S},\rho_{S}), \; \beta_{S}=-\frac{1}{h_{S}}.
			\end{equation*}
		\end{thm}
		\begin{proof}
			See e.g. \cite{17}.
		\end{proof}
		Now we can use Newton polyhedron and the above theorems to derive uniform estimate of some oscillatory integrals that have appeared in the proof of Lemma \ref{3.4}.
		\begin{lemma}\label{-2/3}
			For $k\in \mathbb{N}$, $\xi=(\xi_{1},\cdots,\xi_{d})$, we have
			\begin{equation*}
				M(\xi_{1}^{k})\curlyeqprec(-k,0), \quad M(\xi_{1}^{2}\xi_{2}\pm\xi_{1}\xi_{2}^{2})\curlyeqprec(-\frac{2}{3},0),
			\end{equation*}
			\begin{equation*}
				M(\xi_{1}^{2}\xi_{2}\pm\xi_{2}^{2})\curlyeqprec(-\frac{3}{4},0), \quad M(\xi_{1}^{2}\xi_{2}\pm\xi_{2}^{2}\pm\xi_{1}^{4})\curlyeqprec(-\frac{3}{4},0).
			\end{equation*}
		\end{lemma}
		\begin{proof}
			Using Van der Corput lemma (see e.g. \cite{19}), we can prove the first statement. For the second statement, we have
			\begin{equation*}
				\mathcal{N}(S)=\mathcal{N}(\xi_{1}^{2}\xi_{2}\pm\xi_{1}\xi_{2}^{2})=\left\lbrace(x,y)\in \mathbb{R}_{+}^{2}\Big|x+y>3, x>1, y>1\right\rbrace,
			\end{equation*}
			\begin{equation*}
				\pi=\left\lbrace(x,y)\in \mathbb{R}_{+}^{2}\Big|x+y=3, x\in[1,2]\right\rbrace,\; S_{\pi}=x^{2}y\pm xy^{2},\; d_{S}=\dfrac{3}{2}.
			\end{equation*}
			For the third statement, we have
			\begin{equation*}
			    \mathcal{N}(S)=\mathcal{N}(\xi_{1}^{2}\xi_{2}\pm\xi_{2}^{2})=\left\lbrace(x,y)\in \mathbb{R}_{+}^{2}\Big|x+2y>4, y>1\right\rbrace,
			\end{equation*}
				\begin{equation*}
				\pi=\left\lbrace(x,y)\in \mathbb{R}_{+}^{2}\Big|x+2y=4, x\in[0,2]\right\rbrace,\; S_{\pi}=x^{2}y\pm y^{2},\; d_{S}=\dfrac{4}{3}.
			\end{equation*}
			For the fourth statement, we have 
			\begin{equation*}
				\mathcal{N}(S)=\mathcal{N}(\xi_{1}^{2}\xi_{2}\pm\xi_{2}^{2}\pm \xi_{1}^{4})=\left\lbrace(x,y)\in \mathbb{R}_{+}^{2}\Big|x+2y>4\right\rbrace,
			\end{equation*}
			\begin{equation*}
				\pi=\left\lbrace(x,y)\in \mathbb{R}_{+}^{2}\Big|x+2y=4, x\in[0,4]\right\rbrace,\; S_{\pi}=x^{2}y\pm y^{2}\pm x^{4},\; d_{S}=\dfrac{4}{3}.
			\end{equation*}
			Based on Theorem \ref{superadapted} and Theorem \ref{uniform estimate}, we derive the  statements.
			
		\end{proof}

		\newpage
		\section*{Acknowledgement}
		The author is supported by NSFC, No. 123B1035 and is grateful to Prof. B. Choi and Prof. B. Hua for helpful discussions.
		
		\section*{Conflict of interest statement}
		The author does not have any possible conflict of interest.
		
		\section*{Data availability statement}
		The manuscript has no associated data.
		\bigskip
		\bigskip

		\bibliographystyle{alpha}
		\bibliography{fractional}

\begin{thebibliography}{GZVA12}

\bibitem[BCH23]{12}
Cheng Bi, Jiawei Cheng, and Bobo Hua.
\newblock The wave equation on lattices and oscillatory integrals.
\newblock {\em arXiv preprint arXiv:2312.04130}, 2023.

\bibitem[CA23a]{8}
Brian Choi and Alejandro Aceves.
\newblock Continuum limit of 2d fractional nonlinear schr{\"o}dinger equation.
\newblock {\em Journal of Evolution Equations}, 23(2):30, 2023.

\bibitem[CA23b]{2}
Brian Choi and Alejandro Aceves.
\newblock Continuum limit of 2d fractional nonlinear schr{\"o}dinger equation.
\newblock {\em Journal of Evolution Equations}, 23(2):30, 2023.

\bibitem[DPB93a]{22}
Thierry Dauxois, Michel Peyrard, and AR~Bishop.
\newblock Dynamics and thermodynamics of a nonlinear model for dna
  denaturation.
\newblock {\em Physical Review E}, 47(1):684, 1993.

\bibitem[DPB93b]{23}
Thierry Dauxois, Michel Peyrard, and AR~Bishop.
\newblock Thermodynamics of a nonlinear model for dna denaturation.
\newblock {\em Physica D: Nonlinear Phenomena}, 66(1-2):35--42, 1993.

\bibitem[Eva22]{10}
Lawrence~C Evans.
\newblock {\em Partial differential equations}, volume~19.
\newblock American Mathematical Society, 2022.

\bibitem[Gre07]{18}
Michael Greenblatt.
\newblock Resolution of singularities, asymptotic expansions of integrals over
  sublevel sets, and applications.
\newblock {\em submitted for publication}, 2007.

\bibitem[Gre09]{16}
Michael Greenblatt.
\newblock The asymptotic behavior of degenerate oscillatory integrals in two
  dimensions.
\newblock {\em Journal of Functional Analysis}, 257(6):1759--1798, 2009.

\bibitem[GZVA12]{14}
Sabir~Medzhidovich Gusein-Zade, Aleksander~Nikolaevich Varchenko, and
  VI~Arnold.
\newblock Singularities of differentiable maps.
\newblock In {\em Singularities of Differentiable Maps}, pages 1--492. 2012.

\bibitem[HM15]{7}
Bobo Hua and Delio Mugnolo.
\newblock Time regularity and long-time behavior of parabolic p-laplace
  equations on infinite graphs.
\newblock {\em Journal of Differential Equations}, 259(11):6162--6190, 2015.

\bibitem[HS15]{4}
Younghun Hong and Yannick Sire.
\newblock On fractional schrodinger equations in sobolev spaces.
\newblock {\em arXiv preprint arXiv:1501.01414}, 2015.

\bibitem[HY18]{5}
Younghun Hong and Changhun Yang.
\newblock Uniform strichartz estimates on the lattice.
\newblock {\em arXiv preprint arXiv:1806.07093}, 2018.

\bibitem[HY19]{1}
Younghun Hong and Changhun Yang.
\newblock Strong convergence for discrete nonlinear schro0½4dinger equations
  in the continuum limit.
\newblock {\em SIAM Journal on Mathematical Analysis}, 51(2):1297--1320, 2019.

\bibitem[KA99]{20}
G~Kopidakis and S~Aubry.
\newblock Intraband discrete breathers in disordered nonlinear systems. i.
  delocalization.
\newblock {\em Physica D: Nonlinear Phenomena}, 130(3-4):155--186, 1999.

\bibitem[KA00]{21}
G~Kopidakis and S~Aubry.
\newblock Intraband discrete breathers in disordered nonlinear systems. ii.
  localization.
\newblock {\em Physica D: Nonlinear Phenomena}, 139(3-4):247--275, 2000.

\bibitem[Kar86]{11}
Vladimir~Nikolaevich Karpushkin.
\newblock Uniform estimates of oscillatory integrals with parabolic or
  hyperbolic phase.
\newblock {\em Journal of Soviet Mathematics}, 33:1159--1188, 1986.

\bibitem[Kar02]{17}
Vladimir~Nikolaevich Karpushkin.
\newblock Uniform estimates for oscillatory integrals and volumes with the
  varchenko phase.
\newblock {\em Mathematical Notes}, 72:636--640, 2002.

\bibitem[Kev09]{25}
Panayotis~G Kevrekidis.
\newblock {\em The discrete nonlinear Schr{\"o}dinger equation: mathematical
  analysis, numerical computations and physical perspectives}, volume 232.
\newblock Springer Science \& Business Media, 2009.

\bibitem[KLS13]{3}
Kay Kirkpatrick, Enno Lenzmann, and Gigliola Staffilani.
\newblock On the continuum limit for discrete nls with long-range lattice
  interactions.
\newblock {\em Communications in mathematical physics}, 317(3):563--591, 2013.

\bibitem[KRB01]{24}
PG~Kevrekidis, K{\O}~Rasmussen, and AR~Bishop.
\newblock The discrete nonlinear schr{\"o}dinger equation: a survey of recent
  results.
\newblock {\em International Journal of Modern Physics B}, 15(21):2833--2900,
  2001.

\bibitem[KT98]{6}
Markus Keel and Terence Tao.
\newblock Endpoint strichartz estimates.
\newblock {\em American Journal of Mathematics}, 120(5):955--980, 1998.

\bibitem[Sch98]{9}
Pete Schultz.
\newblock The wave equation on the lattice in two and three dimensions.
\newblock {\em Communications on Pure and Applied Mathematics: A Journal Issued
  by the Courant Institute of Mathematical Sciences}, 51(6):663--695, 1998.

\bibitem[Ste93]{19}
EM~Stein.
\newblock Harmonic analysis: real-variable methods, orthogonality, and
  oscillatory integrals.
\newblock {\em Princeton Mathematical Series}, 43, 1993.

\bibitem[Var76a]{13}
AN~Var{\v{c}}enko.
\newblock Newton polyhedra and estimates of oscillatory integrals.
\newblock {\em Funkcional. Anal. i Prilo{\v{z}}en.}, 10(3):13--38, 1976.

\bibitem[Var76b]{15}
AN~Var{\v{c}}enko.
\newblock Newton polyhedra and estimates of oscillatory integrals.
\newblock {\em Funkcional. Anal. i Prilo{\v{z}}en.}, 10(3):13--38, 1976.

\end{thebibliography}
		
	\end{document}